\documentclass[a4paper]{amsart}
\usepackage[colorlinks,linkcolor=blue,citecolor=blue]{hyperref}
\usepackage{mathrsfs}
\usepackage{bbm}
\usepackage{latexsym, amssymb,amsmath,amsthm}
\usepackage[all, knot]{xy}
\xyoption{arc}
\def \To{\longrightarrow}
\def \dim{\operatorname{dim}}
\def \Q{\operatorname{Q}}
\def \bi{\bowtie}
\def \q{\mathbbm{q}}
\def \dim{\operatorname{dim}}

\def \Hom{\operatorname{Hom}}

\def \ord{\operatorname{ord}}

\def \k{\mathbbm{k}}
\def \C{\mathcal{C}}
\def \D{\Delta}
\def \d{\delta}
\def \dl{\d_{_L}}
\def \dr{\d_{_R}}
\def \e{\varepsilon}
\def \M{\mathrm{M}}
\def \N{\mathbbm{N}}

\def \unit{\mathbbm{1}}
\def \gr{\operatorname{gr}}

\numberwithin{equation}{section}

\newtheorem{theorem}{Theorem}[section]
\newtheorem{lemma}[theorem]{Lemma}
\newtheorem{proposition}[theorem]{Proposition}
\newtheorem{corollary}[theorem]{Corollary}
\newtheorem{definition}[theorem]{Definition}

\newtheorem{remark}[theorem]{Remark}

\newtheorem{problem}[theorem]{Problem}

\begin{document}
\title{Quasi-Frobenius-Lusztig kernels for simple Lie algebras}
\author{Gongxiang Liu}
\address{Department of Mathematics, Nanjing University, Nanjing 210093, China} \email{gxliu@nju.edu.cn}
\author{Fred Van Oystaeyen}
\address{Department of Mathematics and Computer Science, University of Antwerp, Antwerp, Belgium}
 \email{fred.vanoystaeyen@ua.ac.be}
\author{Yinhuo Zhang}
\address{Department of Mathematics and Statistics, University of Hasselt, 3590 Diepenbeek, Belgium}
\email{yinhuo.zhang@uhasselt.be}
\maketitle

\begin{abstract} In \cite{Liu}, the quasi-Frobenius-Lusztig kernel associated with $\mathfrak{sl}_{2}$ was constructed. In this paper
 we construct the  quasi-Frobenius-Lusztig kernel for any simple Lie algebra $\mathfrak{g}$.
\end{abstract}

\section{Introduction}

Considering a product on representations of an algebra, an idea useful in physics, leads to the consideration of a coproduct on the algebra and hence to the study of a bialgebra, or more in particular, a Hopf algebra structure.   The theory of algebraic groups is dual to the theory of commutative and cocommutative Hopf algebras. More general Hopf algebras then fit in a  theory of quantum groups as defined by Drinfeld \cite{Dri1,Dri2}, Jimbo \cite{Ji}, Lusztig \cite{Lus,Lus1} and others. By allowing non-canonical isomorphisms for triple products of representations, leading to so-called associators, one obtains a generalization of a Hopf algebra to a quasi-Hopf algebra, termed quasi-algebra for short. This raises the natural
question whether it is possible to find essentially new quasi-quantum groups corresponding to such quasi-algebras? Now, for a simple finite dimensional Lie algebra $\mathfrak{g}$ over $\mathbbm{C}$, a result
  of Drinfeld \cite[Prop 3.16]{Dri} states that  a quasitriangular quantized quasi-Hopf  enveloping algebra ${U}{\mathfrak{g}[[h]]}$
  is twist equivalent to the usual quantum group $U_{h}\mathfrak{g}$. This means that the quasi-quantum group associated
 to a simple finite dimensional Lie algebra is essentially not  new. But what happens in the restricted case? In other words, does there exist a quasi-algebra analogue for Lusztig's definition of a small quantum group, that is, do we have quasi-Frobenius-Lusztig kernels?

A remarkable recent development in Hopf algebra theory is the Andruskiewitsch-Schneider's classification of finite dimensional pointed Hopf algebras, cf. \cite{AH}; here the Frobenius-Lusztig kernels play a dominant role. So it is reasonable to expect that the theory of quasi-FL kernels (short for Frobenius-Lusztig kernels) will provide insight in the structure of finite dimensional quasi-Hopf algebras. Another direction relates to Conformal Field Theory (CFT).  It has been established by Majid, cf. \cite{Maj1}, that there is a quasitriangular quasi-algebra associated to a Topological Field Theory (TFT, for short). The relevance of quasi-Hopf algebras in TFT has been studied by Dijkgraaf, Pasquier, and Roche \cite{DPR}; in {\it loc.cit}. a new class of semisimple quasitriangular quasi-Hopf algebras, denoted by $D^{\omega}(G)$, has been constructed. Further development in CFT, in particular Logorithmic Conformal Field Theory, are pressing for the systematic construction and deeper study of finite dimensional quasitriangular quasi-algebras, in particular to look at nonsemisimple ones. We refer to \cite{GSTF} and references therein for more detail. It is fair to say that in the present situation there is a lack of such examples.

However,  the answer to the question about the existence of quasi-FL kernels is positive! The simplest quasi-FL kernel has been constructed in \cite{Liu}. It was denoted by ${\Q}\mathbf{u}_{q}(\mathfrak{sl}_{2})$ and it was associated
to $\mathfrak{sl}_{2}$. The aim of the present paper is to define ${\Q}\mathbf{u}_{q}(\mathfrak{g})$, the quasi-FL kernel associated
to an arbitrary simple finite dimensional Lie algebra $\mathfrak{g}$, extending  the ideas found in \cite{Liu}. Inspired by the classical FL kernel theory, one may believe that the quasi-FL kernel associated to a finite dimensional Lie algebra $\mathfrak{g}$ should be the Drinfeld double of the half small quasi-quantum group as defined in \cite{EG1}. Our primary mission
is to compute them and to make a comparison  with the Hopf algebra case. It turns out that the computation of quasi-FL kernels is really much more difficult than in the Hopf case.

Half small quasi-quantum groups appeared in the work of Etingof and Gelaki \cite{EG1} and the notation used for them was $A_{q}(\mathfrak{g})$ where $q$ is an $n^{2}$-th primitive root of unity. In case $n$ is odd and prime to the determinant of the Cartan matrix, they established that $D(A_{q}(\mathfrak{g}))$ is twist equivalent to $\mathbf{u}_{q}(\mathfrak{g})$ \cite{EG} (in case $\mathfrak{g}$ is not of type $G_{2}$). We go on to show that the conditions cannot be removed. More precisely, we establish that $D(A_{q}(\mathfrak{g}))$ is not twist equivalent to any Hopf algebra in many cases. This leads to new examples of nonsemisimple quasitriangular quasi-Hopf algebras and their corresponding braided tensor categories, which have independent interest by themselves.

In Section 2, we include some preliminaries including the definition of $A_{q}(\mathfrak{g})$, some facts about quiver Majid algebras and a useful criterion for a $3$-cocycle to be a $3$-coboundary. The Majid algebra $M_{q}(\mathfrak{g}):=(A_{q}(\mathfrak{g}))^{\ast}$ is studied in detail in Section 3, and we pay particular attention to the Serre relation in Proposition \ref{p3.4}. Section 4 is devoted to the computation of the Drinfeld double $D(A_{q}(\mathfrak{g}))$. The computations are explicit  and some of them are rather tedious. This makes for the technical heart of the paper.

In Section 5 we then go on to provide a presentation for $D(A_{q}(\mathfrak{g}))$ in terms of generators and relations.
 We discover some similarities between $D(A_{q}(\mathfrak{g}))$ and $\mathbf{u}_{q}(\mathfrak{g})$. In particular in Theorem \ref{t5.3},
   we obtain that $D(A_{q}(\mathfrak{g}))\cong{\Q}\mathbf{u}_{q}(\mathfrak{g})$. The final Section 6 is devoted to detecting when ${\Q}\mathbf{u}_{q}(\mathfrak{g})$ is twist equivalent to a Hopf algebra and the cases where this does not happen are identified.

 Throughout, we work over an algebraically closed field $\k$ of characteristic $0$ and $[\frac{\bullet}{\bullet}]$ stands for the floor function, that is, for any natural numbers $a,b$, $[\frac{a}{b}]$ denotes the biggest integer which is not bigger than $\frac{a}{b}$.  For general background knowledge, the reader is referred to \cite{Dri} for quasi-Hopf algebras, to \cite{CE,kassel} for general theory about tensor categories,  and to \cite{HLY} for pointed Majid algebras.

\section{Preliminaries} In this section we will recall the Etingof-Gelaki's constructions of half quasi-quantum groups \cite{EG1},
quiver Majid algebras \cite{H,HLY}, and the Drinfeld double of a quasi-Hopf algebra \cite{FN,Maj,Sch}. Then we
formulate a criterion for a $3$-cocycle of a finite abelian group to
  be a $3$-coboundary.

  \subsection{Half small quasi-quantum group $A_{q}(\mathfrak{g})$.}
  A quasi-bialgebra $(H,\M, \mu, \Delta, \varepsilon, \phi)$ is a
$\k$-algebra $(H,\M,\mu)$ with algebra morphisms $\Delta:\;H\to
H\otimes H$ (the comultiplication) and $\varepsilon:\; H\to
\k$ (the counit), and an invertible element $\phi\in H\otimes
H\otimes H$ (the reassociator), such that:
\begin{gather*}
(id\otimes \Delta)\Delta(a)\phi=\phi(\Delta\otimes id)\Delta(a),\;\;a\in H,\\
(id\otimes id\otimes\Delta)(\phi)(\Delta\otimes id\otimes id)(\phi)=(1\otimes \phi)
(id\otimes \Delta\otimes id)(\phi)(\phi\otimes 1),\\
(\varepsilon\otimes id)\Delta=id=(id\otimes \varepsilon)\Delta,\\
(id\otimes \varepsilon\otimes id)(\phi)=1\otimes 1.
\end{gather*}

Denote $\sum X^{i}\otimes Y^{i}\otimes Z^{i}$ by $\phi$ and
$\sum \overline{X}^{i}\otimes \overline{Y}^{i}\otimes
\overline{Z}^{i}$ by $\phi^{-1}$. A quasi-bialgebra $H$ is called a
quasi-Hopf algebra if there are a linear algebra antimorphism
$S:\;H\to H$ (called the antipode) and two elements
$\alpha,\beta\in H$ satisfying for all $a\in H$:
\begin{gather*}
\sum S(a_{(1)})\alpha a_{(2)}=\alpha \varepsilon(a),\;\;\sum a_{(1)}\beta S(a_{(2)})=\beta \varepsilon(a),\\
\sum X^{i}\beta S(Y^{i})\alpha Z^{i}=1=\sum S(\overline{X^{i}})\alpha \overline{Y^{i}}\beta S(\overline{Z^{i}}).
\end{gather*}
Here and below we use the Sweedler sigma notation $\D(a)=a_{(1)} \otimes
a_{(2)}$ (or $a'\otimes a''$) for the comultiplication and $a_{(1)}\otimes a_{(2)}\otimes \cdots \otimes a_{(n+1)}$
 for the $n$-iterated coproduct $\D^n(a)$ of $a$.
We call an invertible element $J\in H\otimes H$  a (Drinfeld) \emph{twist} of
$H$ if it satisfies $(\varepsilon\otimes id)(J)=(id\otimes
\varepsilon)(J)=1$. For a twist $J=\sum f_{i}\otimes g_{i}$ with
inverse $J^{-1}=\sum \overline{f_{i}}\otimes \overline{g_{i}}$, let:
\begin{equation}\alpha_{J}:= \sum S(\overline{f_{i}})\alpha \overline{g_{i}},\;\;
\beta_{J}:= \sum f_{i}\beta S(g_{i}).
\end{equation}
Given a twist $J$ of $H$, if $\beta_{J}$ is invertible, then one can
construct a new quasi-Hopf algebra structure
$H_{J}=(H,\Delta_{J},\varepsilon,\phi_{J},S_{J},\beta_{J}\alpha_{J},1)$
on the algebra $H$, where:
$$\Delta_{J}(a)=J\Delta(a)J^{-1},\;\;a\in H,$$
$$\phi_{J}=(1\otimes J)(id\otimes \Delta)(J)\phi(\Delta\otimes id)(J^{-1})(J\otimes 1)^{-1}$$
and,
$$S_{J}(a)=\beta_{J}S(a)\beta_{J}^{-1},\;\;a\in H.$$

Next we will define the quasi-Hopf algebra $A_{q}(\mathfrak{g})$.
Given an $m\times m$ Cartan matrix $(a_{ij})$ of finite type, it is known that there is a vector $(d_{1},\ldots,d_{m})$ with integer entries $d_{i}\in \{1,2,3\}$ such that the matrix $(d_{i}a_{ij})$ is symmetric. Let $n\geq 2$ be a natural number and $q$ be an $n^{2}$-th primitive root of unity.

 Let $N, M$, $d\geq 0$ be integers.  Following Gauss, we define:
  $$[N]^{!}_{d}=\prod_{h=1}^{N}\frac{q^{dh}-q^{-dh}}{q^{d}-q^{-d}},\;\;\;\;\left [
\begin{array}{c} M+N\\N
\end{array}\right]_{d}=\frac{[M+N]^{!}_{d}}{[M]^{!}_{d}[N]^{!}_{d}}.$$

 Let $H$ be a finite dimensional Hopf algebra generated by grouplike elements $g_{i}$ and skew-primitive
 elements $e_{i}$, $i=1,\ldots,m$, such that:
 $$g_{i}^{n^{2}}=1,\;\;g_{i}g_{j}=g_{j}g_{i},\;\;g_{i}e_{j}g_{i}^{-1}=q^{\delta_{i,j}}e_{j},\;\;e_{i}^{l_{i}}=0,$$
 $$\sum_{r+s=1-a_{ij}}(-1)^{s}\left [
\begin{array}{c} 1-a_{ij}\\s
\end{array}\right]_{d_{i}}e_{i}^{r}e_{j}e_{i}^{s}=0,\;\;\textrm{if}\;\;i\neq j,$$
and $$\Delta(e_{i})=e_{i}\otimes K_{i}+1\otimes e_{i}$$
where $l_{i}=\textrm{ord}(q^{d_{i}a_{ii}})$, the order of $q^{d_{i}a_{ii}}$, and $K_{i}:=\prod_{j}g_{j}^{d_{i}a_{ij}}$. From now on, we use
\begin{equation}c_{ij}:=d_{i}a_{ij}\end{equation}
to denote the entries of the symmetrized Cartan matrix.

Consider the subalgebra $A\subset H$ generated by $g_{i}^{n}, e_{i}$ for $i=1,\ldots,m$. It is clear that $A$ is not a Hopf subalgebra. However, we will see that it is a quasi-Hopf subalgebra of $H_{J}$ for some twist $J$ of $H$.

Let $\{1_{a}|a=(a_{1},\ldots,a_{m})\in (\mathbbm{Z}_{n^{2}})^{m}\}$ be the set of primitive idempotents of $\k (\mathbbm{Z}_{n^{2}})^{m}$. Define $1^{i}_{k}:=\frac{1}{n^{2}}\sum_{j=0}^{n^{2}-1}(q^{n^{2}-k})^{j}g_{i}^{j}$, and denote by $\epsilon_{i}\in (\mathbbm{Z}_{n^{2}})^{m}$ the vector with $1$ in the $i$-th place and $0$ otherwise. Note that
\begin{equation}1_{a}=1^{1}_{a_{1}}1^{2}_{a_{2}}\cdots 1^{m}_{a_{m}},\;\;\;\;1_{a}g_{i}=q^{a_{i}}1_{a},\;\;\;\;\;1_{a}e_i=e_i1_{a-\epsilon_{i}}.
\end{equation}
Let
$$\q:=q^{n},\;\;h_{i}:=g_{i}^{n}.$$
So the subgroup generated by $h_{i}$ is isomorphic to $(\mathbbm{Z}_{n})^{m}$. Similarly, let $\{\mathbf{1}_{a}|a=(a_{1},\ldots,a_{m})\in (\mathbbm{Z}_{n})^{m}\}$ be the set of primitive idempotents of $\k (\mathbbm{Z}_{n})^{m}$, $\mathbf{1}^{i}_{k}:=\frac{1}{n}\sum_{j=0}^{n-1}(\q^{n-k})^{j}h_{i}^{j}$ and $\epsilon_{i}\in (\mathbbm{Z}_{n})^{m}$ the vector with $1$ in the $i$-th place and $0$ otherwise. For later use, we let $\mathbf{1}_{0}$ stand for
the element $\prod_{i=1}^{m}\mathbf{1}^{i}_{0}$. Then we have the following identities:
\begin{equation} \mathbf{1}^{i}_{k}=\sum_{s=0}^{n-1}1^{i}_{k+sn},\;\;\mathbf{1}_{a}h_{i}=\q^{a_{i}}\mathbf{1}_{a},\;\;\;\;\;\mathbf{1}_{a}e_i=e_i\mathbf{1}_{a-\epsilon_{i}}.
\end{equation}
 For any natural number $x,y$, define $c(x,y):=q^{-x(y-y')}$, where $y'$ denotes the remainder after dividing $y$ by $n$. Let
 \begin{equation} J:=\sum_{a,b\in (\mathbbm{Z}_{n^{2}})^{m}}\prod_{i,j=1}^{m}c(a_{i},b_{j})^{c_{ij}}1_{a}\otimes 1_{b}.
 \end{equation}

Define $d(J):=(1\otimes J)(id\otimes \Delta)(J)(\Delta\otimes id)(J^{-1})(J\otimes 1)^{-1}$, {\it the differential} of $J$.
 The following result is a combination of  Lemma 4.2 and Theorem 4.3 in \cite{EG1}.
 \begin{lemma}\label{l1} \emph{(1)} $d(J)=\sum_{a,b,c\in (\mathbbm{Z}_{n})^{m}}(\prod_{i,j=1}^{m}\q^{-c_{ij}a_{i}[\frac{b_{j}+c_{j}}{n}]})\mathbf{1}_{a}
 \otimes \mathbf{1}_b\otimes \mathbf{1}_{c}.$

 \emph{(2)} The subalgebra $A$ generated by $h_i=g_i^n$ and $e_i, i=1,\cdots, m$, is a quasi-Hopf subalgebra of $H_{J}$ with coproduct $\Delta_{J}$ and reassociator $\phi=d(J)$.
  \end{lemma}

\begin{definition} The quasi-Hopf algebra in Lemma \ref{l1} is called the half small quasi-quantum group of $\mathfrak{g}$, denoted by $A_{q}(\mathfrak{g})$.
  \end{definition}
For simplicity, we introduce two more notations:
   \begin{equation} \label{eq;1}
   \flat_{i}:=\sum_{a\in (\mathbbm{Z}_{n})^{m}}\prod_{j=1}^{m}q^{-c_{ij}a_{j}}\mathbf{1}_{a},\ \ H_i=\prod_{j=1}^{m}h_{j}^{c_{ji}}=\prod_{j=1}^{m}h_{j}^{c_{ij}}.
   \end{equation}
 In \cite{EG,EG1}, there are no explicit formulas for the coproduct, the elements $\alpha, \beta$ and the antipode for $A_{q}(\mathfrak{g})$. In fact, they have the following expressions.

  \begin{lemma}\label{l2}
  For the quasi-Hopf algebra $A_{q}(\mathfrak{g})$, we have, for $i=1,\ldots,m$,
 $$\begin{array}{ll}
 \Delta_{J}(e_{i}) = e_{i}\otimes \flat_{i}^{-1}+1\otimes \sum_{j=1}^{n-1}\mathbf{1}^{i}_{j}e_{i}+H_{i}^{-1}\otimes \mathbf{1}^{i}_{0}e_{i}, &
  \Delta_{J}(h_{i})=h_{i}\otimes h_{i},\\
 \alpha = \sum_{a\in (\mathbbm{Z}_{n})^{m}}\prod_{s,t=1}^{m}\q^{c_{st}a_{s}[\frac{n-1+a_{t}}{n}]}\mathbf{1}_{a},& \beta=1,\\
 S(e_{i})= -(\alpha\sum_{j=1}^{n-1}\mathbf{1}^{i}_{j}e_{i}+H_{i}\alpha \mathbf{1}^{i}_{0}e_{i})\flat_{i}\alpha^{-1},& S(h_{i})=h_{i}^{-1}.
 \end{array}$$
   \end{lemma}
  \begin{proof} We have:
  \begin{eqnarray*}
  \Delta_{J}(e_{i})&=&J\Delta(e_{i})J^{-1}\\
  &=&\sum_{a,b\in (\mathbbm{Z}_{n^{2}})^{m}}\prod_{s,t=1}^{m}c(a_{s},b_{t})^{c_{st}}1_{a}\otimes 1_{b}(e_{i}\otimes K_{i}+1\otimes e_{i})\\&&\times\sum_{c,d\in (\mathbbm{Z}_{n^{2}})^{m}}\prod_{s,t=1}^{m}c(c_{s},d_{t})^{-c_{st}}1_{c}\otimes 1_{d}\\
  &=&\sum_{c,b\in (\mathbbm{Z}_{n^{2}})^{m}}\prod_{s,t=1}^{m}c((c+\epsilon_{i})_{s},b_{t})^{c_{st}}
\prod_{s,t=1}^{m}c(c_{s},b_{t})^{-c_{st}}q^{\sum_{j}c_{ij}b_{j}}1_{c+\epsilon_{i}}e_{i}\otimes 1_{b}\\
  &&+\sum_{a,d\in (\mathbbm{Z}_{n^{2}})^{m}}\prod_{s,t=1}^{m}c(a_{s},(d+\epsilon_{i})_{t})^{c_{st}}
  \prod_{s,t=1}^{m}c(a_{s},d_{t})^{-c_{st}}1_{a}\otimes 1_{d+\epsilon_{i}}e_{i}\\
 &=&\sum_{c,b\in (\mathbbm{Z}_{n^{2}})^{m}}\prod_{s\neq i,t}c(c_{s},b_{t})^{c_{st}-c_{st}}
  \prod_{t=1}^{m}c(c_{i}+1, b_t)^{c_{it}}c(c_{i}, b_t)^{-c_{it}}q^{\sum_{j}c_{ij}b_{j}}1_{c+\epsilon_{i}}e_{i}\otimes 1_{b}\\
  &&+
  \sum_{a,d\in (\mathbbm{Z}_{n^{2}})^{m}}\prod_{s,t\neq i}c(a_{s},d_{t})^{c_{st}-c_{st}}\prod_{s=1}^{m}c(a_{s},d_{i}+1)^{c_{si}}c(a_{s},d_{i})^{-c_{si}}
  1_{a}\otimes 1_{d+\epsilon_{i}}e_{i}\\
 \end{eqnarray*}
\begin{eqnarray*}
   &=&\sum_{c,b\in (\mathbbm{Z}_{n^{2}})^{m}}
  q^{\sum_{j}(-c_{ij}b_{j}+c_{ij}b'_{j})}q^{\sum_{j}c_{ij}b_{j}}1_{c+\epsilon_{i}}e_{i}\otimes 1_{b}\\
  &&+
  \sum_{a,d\in (\mathbbm{Z}_{n^{2}})^{m}}\prod_{s=1}^{m}c(a_{s},d_{i}+1)^{c_{si}}c(a_{s},d_{i})^{-c_{si}}
  1_{a}\otimes 1_{d+\epsilon_{i}}e_{i}\\
  &=&\sum_{a,b\in (\mathbbm{Z}_{n^{2}})^{m}}\prod_{t=1}^{m}q^{c_{it}b_{t}'}1_{a+\epsilon_{i}}e_{i}\otimes 1_{b}+
  \sum_{a,b\in (\mathbbm{Z}_{n^{2}})^{m}}\prod_{s=1}^{m}q^{c_{si}a_{s}((b_{i}+1)'-b_{i}'-1)}1_{a}\otimes 1_{b+\epsilon_{i}}e_{i}\\
  &=&e_{i}\otimes \sum_{b\in (\mathbbm{Z}_{n})^{m}}\prod_{t=1}^{m}q^{c_{it}b_{t}}\mathbf{1}_{b}
  +  \sum_{a,b\in (\mathbbm{Z}_{n})^{m}}\prod_{s=1}^{m}q^{c_{si}a_{s}((b_{i}+1)'-b_{i}'-1)}\mathbf{1}_{a}\otimes \mathbf{1}^{i}_{b_{i}+1}e_{i}\\
  &=&e_{i}\otimes \flat_{i}^{-1}
  +  \sum_{a\in (\mathbbm{Z}_{n})^{m}}\mathbf{1}_{a}\otimes  \sum_{b_{i}\neq n-1}\mathbf{1}^{i}_{b_{i}+1}e_{i}
  +\sum_{a\in (\mathbbm{Z}_{n})^{m}}\prod_{s=1}^{m}\q^{-c_{si}a_{s}}\mathbf{1}_{a}\otimes  \mathbf{1}^{i}_{0}e_{i}\\
  &=& e_{i}\otimes \flat_{i}^{-1}+1\otimes \sum_{j=1}^{n-1}\mathbf{1}^{i}_{j}e_{i}+H_{i}^{-1}\otimes \mathbf{1}^{i}_{0}e_{i}.
  \end{eqnarray*}
  By definition:
  $$\alpha_{J}=\sum S(\overline{f}_{i})\overline{g}_i=\sum_{a\in (\mathbbm{Z}_{n^{2}})^{m}}\prod_{s,t=1}^{m}c(-a_{s},a_{t})^{-c_{st}}1_{a}=
  \sum_{a\in (\mathbbm{Z}_{n^{2}})^{m}}\prod_{s,t=1}^{m}q^{-c_{st}a_{s}(a_{t}-a_{t}')}1_{a},$$
  $$\beta_{J}=\sum {f}_{i}S({g}_i)=\sum_{a\in (\mathbbm{Z}_{n^{2}})^{m}}\prod_{s,t=1}^{m}c(a_{s},-a_{t})^{c_{st}}1_{a}=
  \sum_{a\in (\mathbbm{Z}_{n^{2}})^{m}}\prod_{s,t=1}^{m}q^{-c_{st}a_{s}(-a_{t}-(n^{2}-a_{t})')}1_{a},$$
  and so:
   \begin{eqnarray*} \alpha&=&\alpha_{J}\beta_{J}\\
   &=&\sum_{a\in (\mathbbm{Z}_{n^{2}})^{m}}\prod_{s,t=1}^{m}q^{c_{st}a_{s}(a'_{t}+(n^{2}-a_{t})')}1_{a}\\
   &=&\sum_{a\in (\mathbbm{Z}_{n})^{m}}\prod_{s,t=1}^{m}q^{c_{st}a_{s}(a_{t}+(n-a_{t})')}\mathbf{1}_{a}\\
   &=&\sum_{a\in (\mathbbm{Z}_{n})^{m}}\prod_{s,t=1}^{m}\q^{c_{st}a_{s}[\frac{n-1+a_{t}}{n}]}\mathbf{1}_{a}.
    \end{eqnarray*}

   By the comultiplication formula for $e_{i}$ and the definition of the antipode, we obtain:
   $$S(e_{i})\alpha\flat_{i}^{-1}+\alpha\sum_{j=1}^{n-1}\mathbf{1}^{i}_{j}e_{i}+H_{i}
   \alpha\mathbf{1}^{i}_{0}e_{i}=\alpha\varepsilon(e_{i})=0.$$
It follows that $S(e_{i})=-(\alpha\sum_{j=1}^{n-1}\mathbf{1}^{i}_{j}e_{i}+H_{i}\alpha \mathbf{1}^{i}_{0}e_{i})\flat_{i}\alpha^{-1}$. The formulas for elements $h_i$
   are obvious.
  \end{proof}

\begin{remark} \emph{In \cite{EG1} Etingof and Gelaki used the Cartan matrix $(a_{ij})$ to define the half small quasi-quantum group $A_{q}(\mathfrak{g})$. In order to keep the consistency with Lusztig's definition \cite{Lus1}, we use the symmetrized Cartan matrix $(c_{ij})$ instead of $(a_{ij})$ throughout this paper. To show the difference, we use $A'_{q}(\mathfrak{g})$ to denote Etingof-Gelaki's half small quasi-quantum group. The two are equal in the simply laced case. But, in general,
$A_{q}(\mathfrak{g})\not\cong A'_{q}(\mathfrak{g})$ and they are not even twist equivalent (see Section 6 for the definition of twist equivalence).
The reason is that they have different reassociators. For example, take the Cartan matrix of type $G_{2}$  and assume that they are twist equivalent. Denote by $\phi_{A}$ (resp. $\phi_{A'}$) the reassociator of
$A_{q}(\mathfrak{g})$ (resp. $A'_{q}(\mathfrak{g})$).  If the representation categories of $A_{q}(\mathfrak{g})$ and $A'_{q}(\mathfrak{g})$ are monoidal equivalent, then their tensor subcategories generated by simple objects
are also monoidal equivalent. This implies that $((\k G)^{\ast}, \phi_{A})$ is twist equivalent to $((\k G)^{\ast}, \phi_{A'})$, where $G\cong (\mathbbm{Z}_{n})^{m}$ is the
set of group-like elements of both $A_{q}(\mathfrak{g})$ and $A'_{q}(\mathfrak{g})$. However, $((\k G)^{\ast}, \phi_{A})$ and $((\k G)^{\ast},\phi_{A'})$ are twist equivalent if and only if $\phi_{A}$ and $\phi'_{A}$ are cohomologous cocycles of $G$.  But clearly they are not cohomologous. So $A_{q}(\mathfrak{g})$ and $ A'_{q}(\mathfrak{g})$ are not twist equivalent.}
\end{remark}

\subsection{Quiver Majid algebras.}
A dual quasi-bialgebra, or Majid bialgebra for short, is a coalgebra
$(H,\D,\e)$ equipped with a compatible quasi-algebra structure.
Namely, there exist two coalgebra homomorphisms:
 $$\M: H \otimes H
\to H, \ a \otimes b \mapsto ab, \quad \mu: \k \to H,\ \lambda
\mapsto \lambda 1_H$$ and a convolution-invertible map $\Phi:
H^{\otimes 3} \to \k$ called a reassociator, such that for all $a,b,c,d
\in H$ the following equalities hold:
\begin{gather*}
a_{(1)}(b_{(1)}c_{(1)})\Phi(a_{(2)},b_{(2)},c_{(2)})=\Phi(a_{(1)},b_{(1)},c_{(1)})(a_{(2)} b_{(2)})c_{(2)},\\
1_Ha=a=a1_H, \\
\Phi(a_{(1)},b_{(1)},c_{(1)}d_{(1)})\Phi(a_{(2)}b_{(2)},c_{(2)},d_{(2)}) \\
 =\Phi(b_{(1)},c_{(1)},d_{(1)})\Phi(a_{(1)},b_{(2)}c_{(2)},d_{(2)})\Phi(a_{(3)},b_{(1)},c_{(3)}),\nonumber \\
\Phi(a,1_H,b)=\e(a)\e(b).
\end{gather*}
 $H$ is called a Majid algebra if, in addition,
there exist a coalgebra antimorphism $S: H \to H$ and two
functionals $\alpha,\beta: H \to \k$ such that for all $a \in H,$
\begin{gather*}
S(a_{(1)})\alpha(a_{(2)})a_{(3)}=\alpha(a)1_H, \quad
a_{(1)}\beta(a_{(2)})S(a_{(3)})=\beta(a)1_H, \\
\Phi(a_{(1)},S(a_{(3)}), a_{(5)})\beta(a_{(2)})\alpha(a_{(4)})= \\
\Phi^{-1}(S(a_{(1)}),a_{(3)},S(a_{(5)})) \alpha(a_{(2)})\beta(a_{(4)})=\e(a).
\nonumber
\end{gather*}

A Majid algebra $H$ is said to be pointed, if the underlying
coalgebra is pointed. Given a pointed Majid algebra $(H,\D, \e,
\M, \mu, \Phi,S,\alpha,\beta),$ we let $\{H_n\}_{n \ge 0}$ be its
coradical filtration, and $\gr H = H_0 \oplus H_1/H_0 \oplus H_2/H_1
\oplus \cdots$ the associated graded coalgebra. Then
$\gr H$ possesses an (induced) graded Majid algebra structure. The
corresponding graded reassociator $\gr\Phi$ satisfies
$\gr\Phi(\bar{a},\bar{b},\bar{c})=0$ for all
$\bar{a},\bar{b},\bar{c} \in \gr H$ unless they all lie in $H_0.$
Similar condition holds for $\gr\alpha$ and $\gr\beta.$ In
particular, $H_0$ is a Majid subalgebra and it coincides with the
group algebra $\k G$ of the group $G=G(H),$ the set of group-like
elements of $H.$

Now assume that $H$ is a Majid algebra with reassociator $\Phi.$ A
linear space $M$ is called an $H$-Majid bimodule, if $M$ is an
$H$-bicomodule with structure maps $(\dl,\dr),$ and there
are two $H$-bicomodule morphisms:
$$\mu_{L}: H \otimes M \To M, \ h \otimes m \mapsto h\cdot m, \quad
\mu_{R}: M \otimes H \To M, \ m \otimes h \mapsto m\cdot h$$ such that
for all $g, h \in H, m \in M,$ the following equalities hold:
\begin{gather}
1_H\cdot m=m=m\cdot 1_H,\\
g_{(1)}\cdot(h_{(1)}\cdot m_0)\Phi(g_{(2)},h_{(2)},m_1)=\Phi(g_{(1)},h_{(1)},m^{-1})(g_{(2)}h_{(2)})\cdot m^0,\label{eq;1.5} \\
m_0\cdot(g_{(1)}h_{(1)})\Phi(m_1,g_{(2)},h_{(2)})=\Phi(m^{-1},g_{(1)},h_{(1)})(m^{0} \cdot g_{(2)})\cdot h_{(2)}, \label{eq;1.6}\\
g_{(1)}\cdot(m_0\cdot h_{(1)})\Phi(g_{(2)},m_1,h_{(2)})=\Phi(g_{(1)},m^{-1},h_{(1)})(g_{(2)}\cdot m^0)\cdot h_{(2)},\label{eq;1.7}
\end{gather}
where we use the Sweedler notation:
 $$\dl(m)=m^{-1} \otimes m^0,
\quad \dr(m)=m_0 \otimes m_1$$ for the comodule structure maps. Since we only consider Majid bimodules over $(\k G, \Phi)$, it is
convenient to rewrite formulas \eqref{eq;1.5}-\eqref{eq;1.7}. Assume $M$ is a Majid bimodule over $(\k G, \Phi)$ and so $M=\bigoplus_{g,h \in G}
\ ^gM^h$, where:
$$^gM^h=\{ m \in M \ | \ \d_{_L}(m)=g \otimes m, \ \d_{_R}(m)=m
\otimes h \} \ .$$ Formulas \eqref{eq;1.5}-\eqref{eq;1.7} can be simplified as:
\begin{gather}
e\cdot(f\cdot m)=\frac{\Phi(e,f,g)}{\Phi(e,f,h)}(ef)\cdot m,\label{eq;1.8}\\
(m\cdot e)\cdot f=\frac{\Phi(h,e,f)}{\Phi(g,e,f)}m\cdot (ef),\label{eq;1.9}\\
(e\cdot m)\cdot f=\frac{\Phi(e,h,f)}{\Phi(e,g,f)}e\cdot (m\cdot f),\label{eq;1.10}
\end{gather}
for all $e,f,g,h \in G$ and $m \in \ ^gM^h$.

Now let us recall some basic definitions about quivers.
A quiver is a quadruple $Q=(Q_0,Q_1,s,t),$ where $Q_0$ is the set of
vertices, $Q_1$ is the set of arrows, and $s,t:\ Q_1 \longrightarrow
Q_0$ are two maps assigning respectively the source and the target
for each arrow. A path of length $l \ge 1$ in the quiver $Q$ is a
finitely ordered sequence of $l$ arrows $a_l \cdots a_1$ such that
$s(a_{i+1})=t(a_i)$ for $1 \le i \le l-1.$ By convention, a vertex is
said to be a {\it trivial path} of length $0.$
For a quiver $Q,$ the associated path coalgebra $\k Q$ is the
$\k$-space spanned by the set of paths, with counit and
comultiplication maps defined by $\e(g)=1, \ \D(g)=g \otimes g$ for
each $g \in Q_0,$ and for each nontrivial path $p=a_n \cdots a_1, \
\e(p)=0,$
\begin{equation}
\D(a_n \cdots a_1)=p \otimes s(a_1) + \sum_{i=1}^{n-1}a_n \cdots
a_{i+1} \otimes a_i \cdots a_1 \nonumber + t(a_n) \otimes p \ .
\end{equation}
The lengths of paths give a natural gradation to the path coalgebra.
Let $Q_n$ denote the set of paths of length $n$ in $Q$. Then $\k
Q=\oplus_{n \ge 0} \k Q_n$ and $\D(\k Q_n) \subseteq
\oplus_{n=i+j}\k Q_i \otimes \k Q_j.$ It is clear that $\k Q$ is pointed with
the set of group-likes $G(\k Q)=Q_0,$ and has the following
coradical filtration $$ \k Q_0 \subseteq \k Q_0 \oplus \k Q_1
\subseteq \k Q_0 \oplus \k Q_1 \oplus \k Q_2 \subseteq \cdots.$$
Hence $\k Q$ is coradically graded.

In this paper, we consider a special kind of quiver, that is, a \emph{Hopf quiver} \cite{CR} defined via a group and its ramification datum.
Let $G$ be a group and denote by $\C$ its set of conjugacy classes. A ramification datum $R$ of $G$ is a formal sum
$\sum_{C \in \C}R_CC$ of conjugacy classes with coefficients $R_C$ in
$\mathbbm{N}=\{0,1,2,\cdots\}.$ The corresponding Hopf quiver
$Q=Q(G,R)$ is defined as follows: the set of vertices $Q_0$ is $G,$
and for each $x \in G$ and $c \in C,$ there are $R_C$ arrows going
from $x$ to $cx.$ For example, let $G=\mathbbm{Z}_{n}=\langle g\rangle$ and $R=g$, the corresponding Hopf quiver is:
$$ 
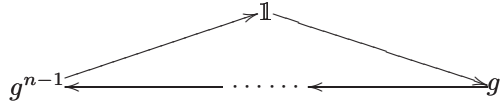
\xy {\ar (0,0)*{\unit}; (30,-10)*{g}}; {\ar (-30,-10)*{g^{n-1}};
(0,0)*{\unit}}; {\ar (30,-10)*{g}; (3,-10)*{ \cdots  \ }}; {\ar
(-3,-10)*{\ \cdots  }; (-30,-10)*{g^{n-1}}}
\endxy $$
It is called a \emph{basic cycle of length $n$}.

It is shown in \cite{qha1} that the path coalgebra $\k Q$ admits a
graded Majid algebra structure if and only if the quiver $Q$ is a
Hopf quiver. Moreover, for a given Hopf quiver $Q=Q(G,R),$ if we fix
a Majid algebra structure on $\k Q_0=(\k G,\Phi)$ with
quasi-antipode $(S,\alpha,\beta),$ then the set of graded Majid
algebra structures on $\k Q$ with $\k Q_0=(\k
G,\Phi,S,\alpha,\beta)$ is in one-to-one correspondence with the
set of $(\k G,\Phi)$-Majid bimodule structures on $\k Q_1.$
We need to recall this correspondence here. One direction is clear. That is,
given a graded Majid algebra structure on the path coalgebra $\k Q$, then  $\k Q_1$ is a $\k Q_{0}$-Majid bimodule with module and comodule structures respectively defined by the multiplication and the comultiplication of $\k Q$.

Conversely, assume that $\k Q_1$ is a $\k Q_{0}$-Majid bimodule. We need to define a multiplication for any two paths in $\k Q$, which can be obtained as follows. Let $p\in \k Q$ be a path. An
{\it $n$-thin split} of $p$ is a sequence $(p_1, \ \cdots, \ p_n)$ of
vertices and arrows such that the concatenation $p_n \cdots p_1$ is
exactly $p.$ These $n$-thin splits are in one-to-one correspondence
with the $n$-sequences of $(n-l)$ 0's and $l$ 1's. Denote  by $D_l^n$ the set of such sequences. Clearly $|D_l^n|={n \choose l}.$ For
$d=(d_1, \ \cdots, \ d_n) \in D_l^n,$ the corresponding $n$-thin
split is written as $dp=((dp)_1, \ \cdots, \ (dp)_n),$ in which
$(dp)_i$ is a vertex if $d_i=0$ and an arrow if $d_i=1.$ Let
$\alpha=a_m \cdots a_1$ and $\beta=b_n \cdots b_1$ be paths of
length $m$ and $n$ respectively. Given $d \in D_m^{m+n}$, we let $\bar{d}
\in D_n^{m+n}$ be the complement sequence of $d$ obtained  by
replacing each 0 by 1 and each 1 by 0. Define an element
$$(\alpha \cdot \beta)_d=[(d\alpha)_{m+n}\cdot(\bar{d}\beta)_{m+n}] \cdots
[(d\alpha)_1\cdot(\bar{d}\beta)_1]$$ in $\k Q_{m+n},$ where
$[(d\alpha)_i\cdot(\bar{d}\beta)_i]$ is understood as the action of the $
\k Q_0$-Majid bimodule on $\k Q_1$ and the terms in different
brackets are put together by cotensor product, or equivalently
by concatenation. In terms of this notation, the formula of the
product of $\alpha$ and $\beta$ is given as follows:
\begin{equation}\label{eq;2}
\alpha \cdot \beta=\sum_{d \in D_m^{m+n}}(\alpha \cdot \beta)_d \ .
\end{equation}

Now let $H=H_{0}\oplus H_{1}\oplus \cdots$ be a coradically graded pointed Majid algebra. The \emph{Gabriel quiver} $Q(H)$ is defined
as follows. Its vertices are group-like elements of $H$, and the number of
arrows between two group-like elements, say $g$ and $h$, is equal to the number of linear independent
non-trivial $(h,g)$-skew primitive elements. Recall that $x$ is an $(h,g)$-skew primitive element if $\Delta(x)=g\otimes x+x\otimes h$
and is \emph{trivial} if $x=c(g-h)$ for some $c\in \k$. The
Gabriel quiver $Q(H)$ possesses the following properties:
\begin{itemize}
\item $Q(H)$ is a Hopf quiver;
\item The  $H_{0}$-Majid bimodule structure
on $H_{1}$  induces a $\k Q(H)_{0}$-Majid bimodule structure on $\k Q(H)_{1}$, and $\k Q(H)$ is hence a Majid algebra;
\item(Theorem of Gabriel's Type) $H$ is a large Majid subalgebra of $\k Q(H)$. By ``a large subalgebra" we mean that it contains the set of vertices and arrows of the Hopf quiver.
\end{itemize}
One may refer to \cite{qha1} for more detail.  The formula \eqref{eq;2} can help us to determine the multiplication of any two elements of $H$. We shall use this observation to study the structure of $A_{q}(\mathfrak{g})^{\ast}$ in the next section.

\subsection{Drinfeld double of a quasi-Hopf algebra.}
The construction of the Drinfeld double of a quasi-Hopf algebra is not a trivial generalization from Hopf to the quasi-Hopf
case. The double of a Hopf algebra $H$ is defined on $H\otimes H^{\ast}$, with $H$ and $H^{\ast}$ being subalgebras. If $H$ is a quasi-Hopf algebra, then
$H^{\ast}$ is not an associative algebra. Thus, one is at a loss to look for an associative algebra structure on $H\otimes H^{\ast}$, and might expect that the double should be some kind of hybrid object. Majid solved this problem in \cite{Maj}. He showed that there exists  a quantum double $D(H)$ as a quasi-Hopf algebra defined on $H\otimes H^*$. Hausser and Nill \cite{FN} gave a computable realization of $D(H)$ on $H\otimes H^{\ast}$. An explicit construction of $D(H)$  was obtained by Schauenburg \cite{Sch}. Here we recall Schauenburg's construction.

Let $(H,\M, \mu, \Delta, \varepsilon, \phi, S, \alpha, \beta)$ be a finite dimensional quasi-Hopf algebra. Let $\phi=\phi^{(1)}\otimes
\phi^{(2)}\otimes \phi^{(3)}=\sum X^{i}\otimes Y^{i}\otimes Z^{i}$ and
$\phi^{-1}=\phi^{(-1)}\otimes
\phi^{(-2)}\otimes \phi^{(-3)}=\sum \overline{X}^{i}\otimes \overline{Y}^{i}\otimes
\overline{Z}^{i}$. Define
\begin{gather}
 \mathbf{\gamma}:=\sum (S(U^{i})\otimes S(T^{i}))(\alpha \otimes \alpha)(V^{i}\otimes W^{i}),\label{eq;1.14}\\
             \mathbf{f}:=\sum (S\otimes S)(\D^{op}(\overline{X}^{i}))\cdot \mathbf{\gamma} \cdot \D(\overline{Y}^{i}\beta S(\overline{Z}^{i})),\label{eq;1.11}\\
              \mathbf{\chi}:=(\phi \otimes 1)(\D\otimes id\otimes id)(\phi^{-1}),\label{eq;1.12}\\
            \mathbf{\omega}:=(1\otimes 1\otimes 1\otimes \tau(\mathbf{f}^{-1}))(id\otimes \D\otimes S\otimes S)(\mathbf{\chi})(\phi\otimes 1\otimes 1),\label{eq;1.15}
\end{gather}
where $(1\otimes \phi^{-1})(id\otimes id\otimes \D)(\phi)=\sum T^{i}\otimes U^{i}\otimes V^{i}\otimes W^{i}$ and $\tau$ is the usual twist, i.e., $\tau(a\otimes b)=b\otimes a$.

As a linear space, $D(H)=H\otimes H^{\ast}$ and we write $h\bowtie \psi:=h\otimes \psi\in D(H)$. There are two canonical actions, denoted by $\rightharpoonup,\;\leftharpoonup$, of $H$
on $H^{\ast}$. By definition, for any $a,b\in H$ and $\psi\in H^{\ast}$
$$\rightharpoonup:\;\;H\otimes H^{\ast}\To H^{\ast}, \;\;\;\;(a\rightharpoonup \psi)(b)=\psi(ba),$$
  $$\leftharpoonup:\;\;H^{\ast}\otimes H\To H^{\ast}, \;\;\;\;( \psi\leftharpoonup a)(b)=\psi(ab).$$
Define a map $\textbf{T}:\;H^{\ast}\to D(H)$ by
\begin{equation}\textbf{T}(\psi)=\phi^{(1)}_{(2)}\bi S(\phi^{(2)})\alpha (\phi^{(3)}\rightharpoonup \psi\leftharpoonup \phi^{(1)}_{(1)}).\label{eq;1.13}
\end{equation}
With the above preparations, we are now able to describe  $D(H)$.

\begin{theorem}\cite[Thm. 6.3, 9.3]{Sch}\label{t1} Let $H$ be a finite dimensional quasi-Hopf algebra. The quasi-Hopf structure on $D(H)=H\otimes H^{\ast}$, such that $H$ is a quasi-Hopf subalgebra via the embedding $h\hookrightarrow h\bi \e$, is determined by the following:
\begin{enumerate}
\item[(i)] As an associative algebra, $D(H)$ is generated by $H$ and $\emph{\textbf{T}}(H^{\ast})$, and its multiplication is given by:
\begin{eqnarray*}&&(g\bi \varphi)(h\bi \psi)\\
&&=gh_{(1)(2)}\omega^{(3)}\bi (\omega^{(5)}\rightharpoonup \psi \leftharpoonup \omega^{(1)})(\omega^{(4)}S(h_{(2)})\rightharpoonup \varphi\leftharpoonup
h_{(1)(1)}\omega^{(2)});\;\;\;(\star)
\end{eqnarray*}
as a quasi-coalgebra, the comultiplication of $D(H)$ is given by:
\begin{eqnarray*}\D_{D}(\mathbf{T}(\psi))&=&
 \tilde{\phi}^{(2)}\mathbf{T}(\psi_{(1)}\leftharpoonup \tilde{\phi}^{(1)})\phi^{(-1)}\phi^{(1)}\\
 &\otimes& \tilde{\phi}^{(3)}
\phi^{(-3)}\mathbf{T}(\phi^{(3)}\rightharpoonup \psi_{(2)}\leftharpoonup \phi^{(-2)})\phi^{(2)},\;\;\;\;(\star\star)\end{eqnarray*}
for $g,h\in H$ and $\varphi,\psi\in H^{\ast}$, where $\tilde{\phi}$ is another copy of $\phi$.\\
\item[(ii)] The reassociator $\phi_{D}$,  the counit $\e_{D}$, the elements $\alpha_{D},\beta_{D}$ and the antipode $S_{D}$ are given by:
\begin{gather} \phi_{D}=\phi\bi \e,
\e_{D}(\mathbf{T}(\psi))=\psi(\phi^{(1)}S(\phi^{(2)})\alpha \phi^{(3)}),\\
\alpha_{D}=\alpha\bi \e,\;\;\;\;\beta_{D}=\beta\bi \e,\\
S_{D}(\mathbf{T}(\psi))=\mathbf{f}^{(2)}\mathbf{T}(\mathbf{f}^{(-2)}\rightharpoonup S^{-1}(\psi)\leftharpoonup \mathbf{f}^{(1)})\mathbf{f}^{(-1)},
\end{gather}
for $\psi\in H^{\ast}$.
\end{enumerate}
\end{theorem}

\begin{remark} \emph{(1)} It is easy to see that $1\bi \e$ is the unit element of $D(H)$ by the formula $(\star)$. Moreover, as a special case of the product, we have:
\begin{equation}\label{eq;4}(1\bi \varphi)(h\bi \e)=h_{(1)(2)}\bi S(h_{(2)})\rightharpoonup \varphi\leftharpoonup h_{(1)(1)},
\end{equation}
for $h\in H$ and $\varphi\in H^{\ast}$.

 \emph{(2)}  In the process of our computations, we find that there are some errors or misprints in \cite{Sch} and \cite{FN}. Especially, there are misprints in the expression of the element $\mathbf{f}$ given both in \cite{Sch} and \cite{FN}, in the expression of the element $\chi$ given in \cite{Sch} and in the expression of the comultiplication formula given in \cite{Sch}. The correct versions are \eqref{eq;1.11}, \eqref{eq;1.12} and \emph{($\star\star$)}.
\end{remark}

\subsection{$3$-cocycles} Let $G$ be a group and $(B_{\bullet},\partial_{\bullet})$ its bar resolution. By applying $\Hom_{\mathbbm{Z}G}(-,\k^{\ast})$
we get a complex $(B^{\ast}_{\bullet},\partial^{\ast}_{\bullet})$, where $\k^{\ast}=\k\setminus \{0\}$ is a trivial
$G$-module. Later on, we will encounter the following problem: Given a $3$-cocycle of the complex $(B^{\ast}_{\bullet},\partial^{\ast}_{\bullet})$,
we have to determine whether it is a $3$-coboundary or not. In this subsection, we solve this problem in case $G$ is a finite abelian group.

Now let $G$ be a finitely generated abelian group. Thus $G\cong \mathbbm{Z}_{m_{1}}\times\cdots \times\mathbbm{Z}_{m_{k}}$.
For every $\mathbbm{Z}_{m_{i}}$, we fix a generator $g_{i}$  throughout of this paper for $1\leq i\leq k$. It is known that the
following periodic sequence is a projective resolution for the
trivial $\mathbbm{Z}_{m_{i}}$-module $\mathbbm{Z}$ \cite[Sec. 6.2]{Wei}:
\begin{equation}\label{eq;3}\cdots\longrightarrow \mathbbm{Z}\mathbbm{Z}_{m_{i}}\stackrel{T_{i}}\longrightarrow
\mathbbm{Z}\mathbbm{Z}_{m_{i}}\stackrel{N_{i}}\longrightarrow\mathbbm{Z}\mathbbm{Z}_{m_{i}}\stackrel{T_{i}}\longrightarrow
\mathbbm{Z}\mathbbm{Z}_{m_{i}}\stackrel{N_{i}}\longrightarrow
\mathbbm{Z}\longrightarrow 0,\end{equation}
where $T_{i}=g_{i}-1$ and $N_{i}=\sum_{j=0}^{m_{i}-1}g_{i}^{j}$.

We construct the tensor product of the above periodic resolutions for
$G$. Let $K_{\bullet}$ be the following complex of projective (in fact,
free) $\mathbbm{Z}G$-modules. For
each sequence  $a_{1},\ldots,a_{k}$ of nonnegative integers, let $\Psi(a_{1},\ldots,a_{k})$ be a free
generator in degree $a_{1}+\cdots+a_{k}$. Define:
$$K_{m}:=\bigoplus_{a_{1}+\cdots+a_{k}=m} (\mathbbm{Z}G)\Psi(a_{1},\ldots,a_{k}),$$
and
$$d_{i}(\Psi(a_{1},\ldots,a_{k}))=\left \{
\begin{array}{lll} 0, &\;\;\;\;a_{i}=0,
\\ (-1)^{\sum_{l<i}a_{l}}N_{i}\Psi(a_{1},\ldots,a_{i}-1,\ldots,a_{k}), &  \ 0\neq a_{i}\ \mathrm{even},\\ (-1)^{\sum_{l<i}a_{l}}T_{i}\Psi(a_{1},\ldots,a_{i}-1,\ldots,a_{k}), & \  a_{i}\ \textrm{odd},
\end{array}\right.$$
for $1\leq i\leq k$. The differential $d$ is defined to be $d_{1}+\cdots +d_{k}$.

\begin{lemma}\label{l1.3} $(K_{\bullet},d)$ is a free resolution of the trivial $\mathbbm{Z}G$-module $\mathbbm{Z}$.
\end{lemma}
\begin{proof} Observer that $(K_{\bullet},d)$ is exactly the tensor product of the complexes \eqref{eq;3}. Thus
the lemma follows from   the K\"unneth formula for
complexes (see (3.6.3) in \cite{Wei}).
\end{proof}

For convenience, we fix the following notations. For any $1\leq r\leq k$, define $\Psi_{r}:=\Psi(0,\ldots,1,\ldots,0)$ where $1$ lies in the $r$-th position.
For any $1\leq r\leq s\leq k$, define $\Psi_{r,s}:=\Psi(0,\ldots,1,\ldots,1,\ldots,0)$ where $1$ lies in both the $r$-th and the $s$-th position if $r<s$ and
 $\Psi_{r,r}:=\Psi(0,\ldots,2,\ldots,0)$ where $2$ lies in the $r$-th position. Similarly, one can define $\Psi_{r,s,t},\Psi_{r,s,s},\Psi_{r,r,s}$ and $\Psi_{r,r,r}$ for $1\leq r\leq k$, $1\leq r<s\leq k$ and $1\leq r<s<r\leq k$. One could  even define $\Psi_{i,j,s,t},\Psi_{i,i,j,s}$, $\Psi_{i,j,s,s}$, $\Psi_{i,j,j,s}$, $\Psi_{i,i,j,j}$, $\Psi_{i,i,i,j}$, $\Psi_{i,j,j,j}$,
 and $\Psi_{i,i,i,i}$  for $1\leq i\leq k$, $1\leq i<j\leq k$, $1\leq i<j<s\leq k$ and $1\leq i<j<s<t\leq k$ respectively.
 Now it it clear that any cochain $f\in \Hom_{\mathbbm{Z}G}(K_{3},\k^{\ast})$ is uniquely determined by its values on $\Psi_{r,s,t},\Psi_{r,s,s},\Psi_{r,r,s}$
 and $\Psi_{r,r,r}$ for $1\leq r\leq k$, $1\leq r<s\leq k$ and $1\leq r<s<t\leq k$.  For such numbers, we let $f_{r,s,t}=f(\Psi_{r,s,t}),f_{r,s,s}=f(\Psi_{r,s,s}),f_{r,r,s}=f(\Psi_{r,r,s})$ and $f_{r,r,r}=f(\Psi_{r,r,r})$.

\begin{lemma}\label{l1.4} The $3$-cochain $f\in \Hom_{\mathbbm{Z}G}(K_{3},\k^{\ast})$ is a cocycle if and only if for all $1\leq r\leq k$, $1\leq r<s\leq k$ and $1\leq r<s<t\leq k$,
\begin{equation}\label{eq;1.4} f_{r,r,r}^{m_{r}}=1,\;\;f_{r,s,s}^{m_{r}}f_{r,r,s}^{m_{s}}=1,\;\;f_{r,s,t}^{m_{r}}=f_{r,s,t}^{m_{s}}=f_{r,s,t}^{m_{t}}=1.
\end{equation}
\end{lemma}
\begin{proof} The proof follows direct computations. By definition, $f$ is a $3$-cocycle if and only if $1=d^{\ast}(f)(\Psi_{i,j,s,t})=f(d(\Psi_{i,j,s,t}))$ for all
 $1\leq i\leq j\leq s\leq t\leq k$.
For any $a\in \k^{\ast}$, it is clear that $T_{i}\cdot a=1$ since $\k^{\ast}$ is considered as a trivial $G$-module. This means that we only need to consider the condition $1=d^{\ast}(f)(\Psi_{i,j,s,t})$
in the cases: $i=j=s=t$, $i=j<s<t$, $i<j=s<t$, $i<j<s=t$ and $i=j<s=t$ respectively. In case $i=j=s=t$, we have $1=d^{\ast}(f)(\Psi_{i,i,i,i})=f(N_{i}\Psi_{i,i,i})=N_{i}\cdot f_{i,i,i}=f_{i,i,i}^{m_{i}}$.
Similarly, if  $i=j<s<t$, we have $f_{i,s,t}^{m_{i}}=1$. If $i<j=s<t$, then we have $f_{i,j,t}^{-m_{j}}=1$. In case $i<j<s=t$, we then have $f_{i,j,s}^{m_{s}}=1$. Finally, if $i=j<s=t$, we have
$f_{i,s,s}^{m_{i}}f_{i,i,s}^{m_{s}}=1$.  Now it is easy to see  that these relations are  the same as in Equation \eqref{eq;1.4}.
\end{proof}

\begin{lemma}\label{l1.5}The $3$-cochain $f\in \Hom_{\mathbbm{Z}G}(K_{3},\k^{\ast})$ is a coboundary if and only if for all $1\leq i<j\leq k$, there are $g_{i,j}\in \k^{\ast}$ such that
\begin{equation} f_{i,i,j}=g_{i,j}^{m_{i}},\;\;f_{i,j,j}=g_{i,j}^{-m_{j}},\ {\rm and}\ \  f_{l,l,l}=1\;\;f_{r,s,t}=1,
\end{equation} for $1\leq l\leq k$, and $1\leq r<s<t\leq k$.
\end{lemma}
\begin{proof} By definition, $f$ is a coboundary if and only if $f=d^{\ast}(g)$ for some $2$-cochain $g\in \Hom_{\mathbbm{Z}G}(K_{2},\k^{\ast})$. For any $1\leq i\leq j\leq k$,
let $g_{i,j}:=g(\Psi_{i,j})$. Since $T_{l}\cdot a=1$ for any $a\in \k^{\ast}$, we have $d^{\ast}(g)(\Psi_{r,s,t})=d^{\ast}(g)(\Psi_{l,l,l})=1$ for $1\leq r<s<t\leq k$ and $1\leq l\leq k$. Now for all $1\leq i<j\leq k$, $f_{i,i,j}=d^{\ast}(g)(\Psi_{i,i,j})=g(N_{i}\Psi_{i,j}+T_{j}\Psi_{i,i})=g_{i,j}^{m_{i}}$ and $f_{i,j,j}=d^{\ast}(g)(\Psi_{i,j,j})=g(T_{i}\Psi_{j,j}-N_{j}\Psi_{i,j})=g_{i,j}^{-m_{j}}$.
\end{proof}

Lemma \ref{l1.5} provides us an easy way to determine when a $3$-cocycle of the complex $(K_{\bullet}^{\ast},d^{\ast})$ is a $3$-coboundary.
For the bar resolution, it is sufficient to give a chain map from $(K_{\bullet},d_{\bullet})$
to $(B_{\bullet},\partial_{\bullet})$. We define the following three  morphisms of $\mathbbm{Z}G$-modules:
\begin{eqnarray*}
F_{1}: &&K_{1}\To B_{1},\;\;\;\;\Psi_r\mapsto [g_r];\\
F_{2}: &&K_{2}\To B_{2},\\
&&\Psi_{r,s}\mapsto [g_r,g_s]-[g_s,g_r],\\
&&\Psi_{r,r}\mapsto \sum_{l=0}^{m_{r}-1}[g_{r}^{l},g_r];\\
F_{3}: &&K_{3}\To B_{3},\\
&&\Psi_{r,s,t}\mapsto[g_r,g_s,g_t]-[g_s,g_r,g_t]-[g_r,g_t,g_s],\\
&&\;\;\;\;\;\;\;\;\;\;\;\;\;\;[g_t,g_r,g_s]+[g_s,g_t,g_r]-[g_t,g_s,g_r],\\
&&\Psi_{r,r,s}\mapsto \sum_{l=0}^{m_{r}-1}([g_{r}^{l},g_r,g_s]-[g_{r}^{l},g_s,g_r]+[g_s,g_{r}^{l},g_r]),\\
&&\Psi_{r,s,s}\mapsto \sum_{l=0}^{m_{s}-1}([g_{r},g_{s}^{l},g_s]-[g_{s}^{l},g_r,g_s]+[g_{s}^{l},g_{s},g_r]),\\
&&\Psi_{r,r,r}\mapsto \sum_{l=0}^{m_{r}-1}[g_{r},g_{r}^{l},g_r],
\end{eqnarray*}
for $0\leq r\leq k$, $0\leq r<s\leq k$ and $0\leq r<s<t\leq k$.

\begin{lemma}\label{l1.6} The following diagram is commutative:

\begin{figure}[hbt]
\begin{picture}(150,50)(50,-40)
\put(0,0){\makebox(0,0){$ \cdots$}}\put(10,0){\vector(1,0){20}}\put(40,0){\makebox(0,0){$K_{3}$}}
\put(50,0){\vector(1,0){20}}\put(80,0){\makebox(0,0){$K_{2}$}}
\put(90,0){\vector(1,0){20}}\put(120,0){\makebox(0,0){$K_{1}$}}
\put(130,0){\vector(1,0){20}}\put(160,0){\makebox(0,0){$K_{0}$}}
\put(170,0){\vector(1,0){20}}\put(200,0){\makebox(0,0){$\mathbbm{Z}$}}
\put(210,0){\vector(1,0){20}}\put(240,0){\makebox(0,0){$0$}}

\put(0,-40){\makebox(0,0){$ \cdots$}}\put(10,-40){\vector(1,0){20}}\put(40,-40){\makebox(0,0){$B_{3}$}}
\put(50,-40){\vector(1,0){20}}\put(80,-40){\makebox(0,0){$B_{2}$}}
\put(90,-40){\vector(1,0){20}}\put(120,-40){\makebox(0,0){$B_{1}$}}
\put(130,-40){\vector(1,0){20}}\put(160,-40){\makebox(0,0){$B_{0}$}}
\put(170,-40){\vector(1,0){20}}\put(200,-40){\makebox(0,0){$\mathbbm{Z}$}}
\put(210,-40){\vector(1,0){20}}\put(240,-40){\makebox(0,0){$0$}}

\put(40,-10){\vector(0,-1){20}}
\put(80,-10){\vector(0,-1){20}}
\put(120,-10){\vector(0,-1){20}}
\put(158,-10){\line(0,-1){20}}\put(160,-10){\line(0,-1){20}}
\put(198,-10){\line(0,-1){20}}\put(200,-10){\line(0,-1){20}}

\put(60,5){\makebox(0,0){$d$}}
\put(100,5){\makebox(0,0){$d$}}
\put(140,5){\makebox(0,0){$d$}}

\put(60,-35){\makebox(0,0){$\partial_{3}$}}
\put(100,-35){\makebox(0,0){$\partial_{2}$}}
\put(140,-35){\makebox(0,0){$\partial_{1}$}}

\put(50,-20){\makebox(0,0){$F_{3}$}}
\put(90,-20){\makebox(0,0){$F_{2}$}}
\put(130,-20){\makebox(0,0){$F_{1}$}}
\end{picture}
\end{figure}
\end{lemma}
\begin{proof} The proof is routine and so we omit it.
\end{proof}

\begin{corollary}\label{c1.7} Let $\phi\in B_{3}^{\ast}$ be a $3$-cocycle. Then $\phi$ is a
$3$-coboundary if and only if $F_{3}^{\ast}(\phi)$ is a $3$-coboundary.\end{corollary}
\begin{proof} Follows from the fact that $F_{3}^{\ast}$ induces an isomorphism between $3$-cohomology groups.
\end{proof}

\section{The Majid algebra $M_{q}(\mathfrak{g})$}
In this section, we characterize the structure of the Majid algebra $M_{q}(\mathfrak{g}):=A_{q}(\mathfrak{g})^{\ast}$, the dual of $A_{q}(\mathfrak{g})$. It is clear that $M_{q}(\mathfrak{g})$ is a coradically graded pointed Majid algebra
such that the reassociator $\Phi$ is concentrated on
$M_{q}(\mathfrak{g})_{0}$, that is, $\Phi(x,y,z)=0$ unless the homogeneous elements $x,y,z$ all lie in $M_{q}(\mathfrak{g})_{0}$.

Recall that we used $e_{i},h_{i}\;(1\leq i\leq m)$ to denote the generators of $A_{q}(\mathfrak{g})$. It is not hard to see that the elements in
$\{\mathbf{1}_{a}e_{i}^{n_{i}}|a\in (\mathbbm{Z}_{n})^{m},0\leq n_{i}<l_{i},1\leq i\leq m\}$ are linear independent (in fact, $A_{q}(\mathfrak{g})$ is a subalgebra of $\mathbf{u}_{q}(\mathfrak{b})$, where $\mathfrak{b}$ is the Borel subalgebra of $\mathfrak{g}$) and can be extended to a basis
$\{\mathbf{x}_{j}\}_{j}$ consisting of homogeneous elements. The dual basis is denoted by $\{(\mathbf{x}_{j})^{\ast}\}_{j}$.

We first fix some notations. Let $\chi_{i}$ be the character of the group generated by elements $h_{1},\ldots,h_{m}$, defined as follows:
$$\chi_i(h_j):=\q^{\delta_{ij}}.$$ Therefore, $\chi_i=(\mathbf{1}_{\epsilon_{i}})^{\ast}$.
For $a=(a_{1},\ldots,a_{m})\in (\mathbbm{Z}_{n})^{m}$, define $\chi_a:=\prod_{i=1}^{m}\chi_{i}^{a_{i}}$. For $1\leq i\leq m$, and let
$$\Gamma^{i}=(\mathbf{1}_{\epsilon_{i}}e_{i})^{\ast}.$$

\begin{lemma}\label{l3.1} In $M_{q}(\mathfrak{g})$, we have:
\begin{equation}\label{eq;3.1}  \Delta(\Gamma^{j})=\chi_{j}\otimes \Gamma^{j}+\Gamma^{j}\otimes 1,\;\;\;\;(\chi_{i}\Gamma^{j})\chi_{i}^{-1}=\q^{c_{ji}}
q^{-c_{ji}}\Gamma^{j}.
\end{equation}
\end{lemma}
\begin{proof} Note that $(\mathbf{1}_{\epsilon_{j}})^{\ast}=\chi_j$, $\mathbf{1}_{\epsilon_{j}}e_{j}=
e_{j}\mathbf{1}_{{0}}=\mathbf{1}_{\epsilon_{j}}e_{j}\mathbf{1}_{{0}}$ and $(\mathbf{1}_{\epsilon_{j}})^{2}e_{j}=
\mathbf{1}_{\epsilon_{j}}e_{j}$, so we have the first identity.  By the expression of the comultiplication  of $e_{j}$ given in Lemma \ref{l2}, we obtain
$\chi_{i}\Gamma^{j}=(\mathbf{1}_{\epsilon_{i}+\epsilon_{j}}e_{j})^{\ast}$ and $(\mathbf{1}_{\epsilon_{i}+\epsilon_{j}}e_{j})^{\ast}\chi_{i}^{-1}=
q^{c_{ji}(n-1)}(\mathbf{1}_{\epsilon_{j}}e_{j})^{\ast}=\q^{c_{ji}}
q^{-c_{ji}}\Gamma^{j}$.
\end{proof}

Next, we want to obtain the Gabriel quiver of $M_{q}(\mathfrak{g})$.
We denote this quiver by $Q(M)$. It is not hard to determine the set of
vertices of $Q(M)$. Observe that the coradical of $M_{q}(\mathfrak{g})$ equals $(\k G, \Phi)$ where
   $G=\langle \chi_{i}|1\leq i\leq m\rangle \cong (\mathbbm{Z}_{n})^{m}$ and we have:
   $$\Phi(\chi_{a},\chi_b,\chi_c)=\prod_{s,t=1}^{m}\q^{-c_{st}a_{s}[\frac{b_{t}+c_{t}}{n}]}$$ for
   $a,b,c\in (\mathbbm{Z}_{n})^{m}$. Therefore, $Q(M)_0=G$.
   For $1\leq i\leq m$ and $a=(a_{1},\ldots,a_{m})\in (\mathbbm{Z}_{n})^{m}$,
define:
$$\Gamma^{i}_{\chi_{a}}:=\chi_{a}\cdot \Gamma^{i}$$
where $`\cdot$' is the multiplication of $M_{q}(\mathfrak{g})$. By Lemma \ref{l3.1}, $\Gamma^{i}_{\chi_{a}}$ is a non-trivial $(\chi_{a},\chi_{a+\epsilon_{i}})$-
skew primitive element. Clearly $\#\{\Gamma^{i}_{\chi_{a}}|1\leq i\leq m,a\in (\mathbbm{Z}_{n})^{m}\}=mn^{m}$, which equals $\dim J/J^{2}$
where $J$ is the Jacobson radical of $A_{q}(\mathfrak{g})$. The dual relation between the coradical of $M_{q}(\mathfrak{g})$ and the
radical of $A_{q}(\mathfrak{g})$ guarantees that the
set $\{\Gamma^{i}_{\chi_{a}}|1\leq i\leq m,a\in (\mathbbm{Z}_{n})^{m}\}$ forms a basis of $M_{q}(\mathfrak{g})_{1}$, leading to the following description
of $Q(M)_{1}$:
there is an arrow from $\chi_{a}$ to $\chi_{b}$ if and only if $b=a+\epsilon_{i}$ for some $1\leq i\leq m$. In this case, the only
arrow is $\Gamma^{i}_{\chi_{a}}$. Therefore, locally the quiver $Q(M)$ looks like:
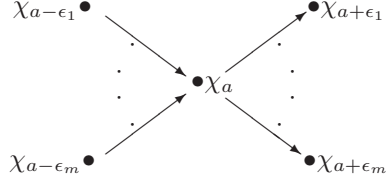
\begin{figure}[hbt]
\begin{picture}(100,60)(-30,-30)
\put(0,0){\makebox(0,0){$\bullet \chi_{a}$}} \put(50,28){\makebox(0,0){$\bullet \chi_{a+\epsilon_{1}}$}}
\put(50,-30){\makebox(0,0){$\bullet \chi_{a+\epsilon_{m}}$}}
 \put(-60,28){\makebox(0,0){$ \chi_{a-\epsilon_{1}}\bullet$}}
\put(-60,-30){\makebox(0,0){$\chi_{a-\epsilon_{m}}\bullet$}}
\put(5,5){\vector(4,3){30}}\put(5,-5){\vector(4,-3){28}}
  \put(25,15){\makebox(0,0){$\cdot$}}  \put(30,5){\makebox(0,0){$\cdot$}}
    \put(30,-5){\makebox(0,0){$\cdot$}}  \put(25,-15){\makebox(0,0){$\cdot$}}
    \put(-40,26){\vector(4,-3){30}}
    \put(-40,-26){\vector(4,3){30}}
  \put(-30,15){\makebox(0,0){$\cdot$}}
   \put(-35,5){\makebox(0,0){$\cdot$}}
     \put(-35,-5){\makebox(0,0){$\cdot$}}
       \put(-30,-15){\makebox(0,0){$\cdot$}}
\end{picture}
\end{figure}

 As stated in Subsection 2.2, there is a $\k Q(M)_0$-Majid bimodule structure on $\k Q(M)_1$ (from
 the $M_{q}(\mathfrak{g})_{0}$-Majid bimodule structure on $M_{q}(\mathfrak{g})_{1}$), which can be described in the following way. We
 will use the equations \eqref{eq;1.8}-\eqref{eq;1.10} freely.

 \begin{lemma}\label{l3.2} The $\k Q(M)_0$-Majid bimodule structure on $\k Q(M)_1$ is given by:
 \begin{eqnarray}
&&\dl:\;\k Q(M)_1\to \k Q(M)_0\otimes \k Q(M)_1,\;\;\Gamma^{j}_{\chi_{a}}\mapsto \chi_{a+\epsilon_{j}} \otimes \Gamma^{j}_{\chi_{a}},\\
&&\dr:\;\k Q(M)_1\to \k Q(M)_1\otimes \k Q(M)_0,\;\;\Gamma^{j}_{\chi_{a}}\mapsto \Gamma^{j}_{\chi_{a}}\otimes \chi_{a},\\
&&\mu_{L}:\; \k Q(M)_0\otimes  \k Q(M)_1\to  \k Q(M)_1,\;\;\chi_{a}\otimes \Gamma^{j}_{\chi_{b}}\mapsto
\prod_{i=1}^{m}\q^{-c_{ij}a_{i}[\frac{b_{j}+1}{n}]}\Gamma^{j}_{\chi_{a+b}},\label{eq;3.4}\\
&&\mu_{R}:\; \k Q(M)_1\otimes  \k Q(M)_0\to  \k Q(M)_1,\;\;\Gamma^{j}_{\chi_{b}}\otimes \chi_{a}\mapsto
\prod_{i=1}^{m}q^{c_{ji}a_{i}}\Gamma^{j}_{\chi_{a+b}},
\end{eqnarray}
for $1\leq j\leq m,\;a,b\in (\mathbbm{Z}_{n})^{m}$, where $a+b$ is understood as the addition in $(\mathbbm{Z}_{n})^{m}$.
 \end{lemma}
\begin{proof} The bicomodule structure is clear since it is obtained from the comultiplication of $M_{q}(\mathfrak{g})$ and
Lemma \ref{l3.1}. Due to our choice, the proof of the left module structure is not complicated:
\begin{eqnarray*}
\chi_{a}\cdot \Gamma^{j}_{\chi_{b}}&=& \chi_{a}\cdot(\chi_{b}\cdot \Gamma^{j})\\
&=&\frac{\Phi(\chi_{a},\chi_{b},\chi_j)}{\Phi(\chi_{a},\chi_b,1)}(\chi_a\chi_b)\cdot \Gamma^{j}\\
&=&\prod_{i=1}^{m}\q^{-c_{ij}a_{i}[\frac{b_{j}+1}{n}]}\Gamma^{j}_{\chi_{a+b}}
\end{eqnarray*}
where the second equality comes from the definition of Majid bimodule (see Equation \eqref{eq;1.8}). In the last equality, we made use of our choice, that is,
$\chi_{a+b}\cdot\Gamma^{j}=\Gamma^{j}_{\chi_{a+b}}$. We divide the proof of the right module structure into four claims.\\[1.5mm]
\textbf{Claim 1:} $(\chi_{i}\cdot \Gamma^{j}_{\chi_{b}})\cdot\chi_{i}^{-1}=\q^{-c_{ij}[\frac{b_{j}+1}{n}]}\q^{c_{ji}}q^{-c_{ji}}\Gamma^{j}_{\chi_{b}}$.

\emph{Proof of Claim 1.} We have \begin{eqnarray*}
(\chi_{i}\cdot \Gamma^{j}_{\chi_{b}})\cdot\chi_{i}^{-1}&=&\frac{\Phi(\chi_i,\chi_{b},\chi_j)}{\Phi(\chi_{i},\chi_{b},1)}((\chi_{b}\chi_{i})\cdot \Gamma^{j})\cdot\chi_{i}^{-1}\\
&=&\frac{\Phi(\chi_i,\chi_{b},\chi_j)}{\Phi(\chi_{i},\chi_{b},1)}\frac{\Phi(\chi_b,\chi_{i},1)}{\Phi(\chi_{b},\chi_{i},\chi_{j})}
(\chi_{b}\cdot(\chi_{i}\cdot \Gamma^{j}))\cdot\chi_{i}^{-1}\\
&=&\frac{\Phi(\chi_i,\chi_{b},\chi_j)}{\Phi(\chi_{b},\chi_{i},\chi_{j})}
\frac{\Phi(\chi_b,\chi_{i},\chi_{i}^{-1})}{\Phi(\chi_{b},\chi_{i}\chi_{j},\chi_{i}^{-1})}\chi_{b}\cdot((\chi_{i}\cdot \Gamma^{j})\cdot\chi_{i}^{-1})\\
&=&\q^{-c_{ij}[\frac{b_{j}+1}{n}]}\q^{c_{ji}}q^{-c_{ji}}\chi_{b}\cdot \Gamma^{j}\\
&=&\q^{-c_{ij}[\frac{b_{j}+1}{n}]}\q^{c_{ji}}q^{-c_{ji}}\Gamma^{j}_{\chi_{b}},
\end{eqnarray*}
where the fourth equality follows from $\frac{\Phi(\chi_i,\chi_{b},\chi_j)}{\Phi(\chi_{b},\chi_{i},\chi_{j})}
\frac{\Phi(\chi_b,\chi_{i},\chi_{i}^{-1})}{\Phi(\chi_{b},\chi_{i}\chi_{j},\chi_{i}^{-1})}=\q^{-c_{ij}[\frac{b_{j}+1}{n}]}$ and Lemma \ref{l3.1}.\\[1.5mm]
\textbf{Claim 2:} $\Gamma^{j}_{\chi_{b}}\cdot \chi_i=q^{c_{ji}}\Gamma^{j}_{\chi_{b+\epsilon_{i}}}$.

\emph{Proof of Claim 2.} Since $\Delta$ is an algebra morphism, $\Gamma^{j}_{\chi_{b}}\cdot \chi_i$ is a
$(\chi_{b+\epsilon_{i}},\chi_{b+\epsilon_{i}+\epsilon_j})$-skew primitive element. So there is a scalar $c\in \k$ such
that $\Gamma^{j}_{\chi_{b}}\cdot \chi_i=c\Gamma^{j}_{\chi_{b+\epsilon_{i}}}$ since the space of non-trivial $(\chi_{b+\epsilon_{i}},\chi_{b+\epsilon_{i}+\epsilon_j})$-skew primitive elements in $M_{q}(\mathfrak{g})$ is $1$-dimensional.
We show that $c=q^{c_{ji}}$. In fact, we have
\begin{eqnarray*}
\Gamma^{j}_{\chi_{b}}&=&\Gamma^{j}_{\chi_{b}}\cdot \chi_{i}^{n}=\frac{\Phi(\chi_{b+\epsilon_{j}},\chi_{i},\chi_{i}^{n-1})}
{\Phi(\chi_{b},\chi_{i},\chi_{i}^{n-1})}(\Gamma^{j}_{\chi_{b}}\cdot \chi_{i})\cdot \chi_{i}^{n-1}\\
&=&\q^{-c_{ji}}c\Gamma^{j}_{\chi_{b+\epsilon_{i}}}\cdot \chi_{i}^{n-1}\\
&=&\q^{-c_{ji}}c\q^{c_{ij}[\frac{b_{j}+1}{n}]}(\chi_{i}\cdot \Gamma^{j}_{\chi_{b}})\cdot \chi_{i}^{-1}\\
&=&cq^{-c_{ji}}\Gamma^{j}_{\chi_{b}},
\end{eqnarray*}
where the third equality follows from \eqref{eq;3.4} and the last equality follows from Claim 1. Therefore, $c=q^{c_{ji}}$.\\[1.5mm]
\textbf{Claim 3:}  $\Gamma^{j}_{\chi_{b}}\cdot \chi_{i}^{l}=q^{c_{ji}l}\Gamma^{j}_{\chi_{b+l\epsilon_{i}}}$ for $0\leq l< n$.

\emph{Proof of Claim 3.} Clearly, one can assume that $l\geq 1$. Inductively, we assume that $\Gamma^{j}_{\chi_{b}}\cdot \chi_{i}^{l-1}=q^{c_{ji}(l-1)}\Gamma^{j}_{\chi_{b+(l-1)\epsilon_{i}}}$ for any $b\in (\mathbbm{Z}_{n})^{m}$. Then
\begin{eqnarray*}
\Gamma^{j}_{\chi_{b}}\cdot \chi_{i}^{l}&=&\Gamma^{j}_{\chi_{b}}\cdot (\chi_{i} \chi_{i}^{l-1})\\
&=&\frac{\Phi(\chi_{b+\epsilon_{j}},\chi_{i},\chi_{i}^{l-1})}
{\Phi(\chi_{b},\chi_{i},\chi_{i}^{l-1})} (\Gamma^{j}_{\chi_{b}}\cdot \chi_i)\cdot \chi_{i}^{l-1}\\
&=&q^{c_{ji}}\Gamma^{j}_{\chi_{b+\epsilon_{i}}}\cdot \chi_{i}^{l-1}\\
&=& q^{c_{ji}l}\Gamma^{j}_{\chi_{b+l\epsilon_{i}}},
\end{eqnarray*}
where in the last equality we used the induction hypothesis.\\[1.5mm]
\textbf{Claim 4:}  $\Gamma^{j}_{\chi_{b}}\cdot \chi_{a}=\prod_{i=1}^{m}q^{c_{ji}a_{i}}\Gamma^{j}_{\chi_{a+b}}$ for $a=(a_{1},\ldots,a_{m})$.

\emph{Proof of Claim 4.}  For $1\leq s\neq t\leq m$ and $0\leq c_{s},c_{t}\leq n-1$, we have
\begin{eqnarray*}\Gamma^{j}_{\chi_{b}}\cdot (\chi^{c_{s}}_{s}\chi_{t}^{c_{t}})&=&
\frac{\Phi(\chi_{b+\epsilon_j},\chi^{c_{s}}_{s},\chi^{c_{t}}_{t})}{\Phi(\chi_{b},\chi^{c_{s}}_{s},\chi^{c_{t}}_{t})} (\Gamma^{j}_{\chi_{b}}\cdot \chi^{c_{s}}_{s})\cdot \chi_{t}^{c_{t}}\\
&=&(\Gamma^{j}_{\chi_{b}}\cdot \chi^{c_{s}}_{s})\cdot \chi_{t}^{c_{t}}.
\end{eqnarray*}
That is, the right module structure is associative in case $s\neq t$. In general, one can repeat the above proof to show that:
$$\Gamma^{j}_{\chi_{b}}\cdot (\chi_{1}^{a_1}\chi_{2}^{a_2}\cdots \chi_{m}^{a_m})=(\ldots((\Gamma^{j}_{\chi_{b}}\cdot \chi_{1}^{a_1})
\cdot \chi_{2}^{a_2})\cdot \cdots)\cdot \chi_{m}^{a_m}).$$
Together with Claim 3, this gives the proof of Claim 4.
\end{proof}

The above lemma and the formula \eqref{eq;2} provide us a helpful tool to determine the relations of  the generators of $M_{q}{(\mathfrak{g})}$.
Since the multiplication of $M_{q}{(\mathfrak{g})}$ is not associative in general, we need to put parentheses in products.
We  define:
$$X^{\stackrel{\rightharpoonup}{l}}:=\stackrel{l}{\overbrace{(\cdots(X\cdot X)\cdot X)\cdots \cdot X)}},\;\;X^{\stackrel{\leftharpoonup}{l}}:=\stackrel{l}{\overbrace{(X\cdot\cdots(X\cdot(X\cdot X))
\cdots)}},$$ for any $X\in M_{q}{(\mathfrak{g})}$.

\begin{proposition}\label{p3.3} For $1\leq i\leq m$, let $l_{i}=\ord(q^{c_{ii}})$. Then we have
\begin{equation}\label{eq;3.3} (\Gamma^{i})^{\stackrel{\rightharpoonup}{l_{i}}}=(\Gamma^{i})^{\stackrel{\leftharpoonup}{l_{i}}}=0
\end{equation}
and $(\Gamma^{i})^{\stackrel{\rightharpoonup}{l}}\neq 0\neq (\Gamma^{i})^{\stackrel{\leftharpoonup}{l}}$ for $l< l_{i}$.
\end{proposition}
\begin{proof} The proof of this result is parallel with the proof of \cite[Lem. 3.6]{HLY} and so we omit the computation. We just explain
why the proof of \cite[Lem. 3.6]{HLY} can apply to our case and what results the computation will deliver. Let $Q^{i}$ be the subquiver of $Q(M)$ defined as follows:
  the set of vertices is $Q^{i}_0=\langle \chi_i \rangle\cong \mathbbm{Z}_{n}$ and the set of arrows is $Q^{i}_1=\{\chi_{i}^{k}\cdot\Gamma^{i}|0\leq k\leq n-1\}$.
  Clearly, $Q^{i}$ is a basic cycle of length $n$, which is the case considered in \cite[Lem. 3.6]{HLY}. By the formula \eqref{eq;2}, we find that, to
  compute $(\Gamma^{i})^{\stackrel{\rightharpoonup}{l}}$ and $(\Gamma^{i})^{\stackrel{\leftharpoonup}{l}}$, we only need to consider the Majid subalgebra
  $\k Q^{i}$. Therefore, \cite[Lem. 3.6]{HLY} applies. Moreover, if we let $p_{t_{i}}^{l}$ be the path starting from $\chi_{i}^{t}$ with length
  $l$ in $ Q^{i}$,  then \cite[Lem. 3.6]{HLY} implies that
\begin{equation}\label{eq;3.3} (\Gamma^{i})^{\stackrel{\rightharpoonup}{l}}=l!_{q^{c_{ii}}}p_{0_{i}}^{l},\;\;
  (\Gamma^{i})^{\stackrel{\leftharpoonup}{l}}=
  \q^{-c_{ii}l'[\frac{l}{n}]}l!_{q^{c_{ii}}}p_{0_{i}}^{l},\end{equation}
  where  $l!_{q^{c_{ii}}}=\sum_{j=0}^{l-1}q^{jc_{ii}}$ by definition, and $l'$ is the remainder after dividing $l$ by $n$.
  As a consequence, we obtain the desired equations.
\end{proof}

In the following conclusion, there is a delicate point at notation: We will use $\Gamma^{j}_{\chi_i}\cdot \Gamma^{i}$ to
denote the multiplication in $\k Q(M)$ while $\Gamma^{j}_{\chi_i} \Gamma^{i}$  stands for the connection of arrows (that is, $\Gamma^{j}_{\chi_i} \Gamma^{i}$ is the path
$1\to \chi_i\to \chi_i\chi_j$ in $Q(M)$).

\begin{proposition}\label{p3.4} Assume $n\geq4$. Then for $1\leq i\neq j\leq m$,  we have the following Serre relation:
 $$\sum_{r+s=1-a_{ij}}(-1)^{s}\left [
\begin{array}{c} 1-a_{ij}\\s
\end{array}\right]_{d_{i}}((\Gamma^i)^{\stackrel{\rightharpoonup}{r}}\cdot \Gamma^{j})\cdot (\Gamma^{i})^{\stackrel{\leftharpoonup}{s}}=0.$$
\end{proposition}
\begin{proof} Since $n\geq 4$, Formula \eqref{eq;3.3} implies:
$$(\Gamma^i)^{\stackrel{\rightharpoonup}{r}}=(\Gamma^i)^{\stackrel{\leftharpoonup}{r}},\;\;\;(0\leq r\leq 1-a_{ij}).$$
So there is no harm in writing $(\Gamma^i)^{r}$ for both $(\Gamma^i)^{\stackrel{\rightharpoonup}{r}}$ and $(\Gamma^i)^{\stackrel{\leftharpoonup}{r}}$.
Moreover, by the definition of a Majid algebra, we have
$$((\Gamma^i)^{r}\cdot \Gamma^{j})\cdot (\Gamma^i)^{s}=\frac{\Phi(1,1,1)}{\Phi(\chi_{i}^r,\chi_j,\chi_{i}^s)}(\Gamma^i)^{r}\cdot (\Gamma^{j}\cdot (\Gamma^i)^{s})=(\Gamma^i)^{r}\cdot (\Gamma^{j}\cdot (\Gamma^i)^{s}).$$
Therefore, the above Serre relation can be written in a more familiar form:
 $$\sum_{r+s=1-a_{ij}}(-1)^{s}\left [
\begin{array}{c} 1-a_{ij}\\s
\end{array}\right]_{d_{i}}(\Gamma^i)^{r}\cdot \Gamma^{j}\cdot (\Gamma^{i})^{s}=0.$$

We will provide a detailed proof of the Serre  relation only for  Lie algebra $\mathfrak{g}$ of type ADE, where we find out how the product
formula \eqref{eq;2} can be used. For the other types of Lie algebras, we just state the computating results.
By \eqref{eq;2}, we have:
\begin{eqnarray*}
\Gamma^{i}\cdot \Gamma^{j}&=&[\Gamma^{i}\cdot \chi_j][1\cdot \Gamma^{j}]+[\chi_i \cdot \Gamma^{j}][\Gamma^{i}\cdot 1]\\
&=&q^{c_{ij}}\Gamma^{i}_{\chi_{j}}\Gamma^{j}+\Gamma^{j}_{\chi_{i}}\Gamma^{i}
\end{eqnarray*}
and
\begin{eqnarray*}
\Gamma^{j}\cdot \Gamma^{i}&=&[\Gamma^{j}\cdot \chi_i][1\cdot \Gamma^{i}]+[\chi_j \cdot \Gamma^{i}][\Gamma^{j}\cdot 1]\\
&=&q^{c_{ji}}\Gamma^{j}_{\chi_{i}}\Gamma^{i}+\Gamma^{i}_{\chi_{j}}\Gamma^{j}.
\end{eqnarray*}
Here $\Gamma^{j}_{\chi_{i}}\Gamma^{i}$ is the path $1\to \chi_i\to \chi_i\chi_j$
 and $\Gamma^{i}_{\chi_{j}}\Gamma^{j}$ is the path $1\to \chi_j\to \chi_i\chi_j$. Therefore, if $a_{ij}=0$ (which implies $c_{ij}=0$) then we have $\Gamma^{i}\cdot \Gamma^{j}=\Gamma^{j}\cdot \Gamma^{i}$. Now
consider the case $a_{ij}=-1$. We compute:
\begin{eqnarray*}
\Gamma^{i}\cdot \Gamma^{i}_{\chi_{j}}\Gamma^{j}&=& [\Gamma^{i}\cdot \chi_i\chi_j][1\cdot \Gamma^{i}_{\chi_j}][1\cdot \Gamma^{j}]+
[\chi_{i}\cdot \Gamma^{i}_{\chi_j}][\Gamma^{i}\cdot {\chi_j}][1\cdot \Gamma^{j}]\\
&&+[\chi_{i}\cdot \Gamma^{i}_{\chi_j}][\chi_{i}\cdot\Gamma^{j}][\Gamma^{i}\cdot 1]\\
&=& q^{c_{ii}+c_{ij}}\Gamma^{i}_{\chi_i\chi_j}\Gamma^{i}_{\chi_j}\Gamma^{j}+q^{c_{ij}}\Gamma^{i}_{\chi_i\chi_j}\Gamma^{i}_{\chi_j}\Gamma^{j}
+\Gamma^{i}_{\chi_i\chi_j}\Gamma^{j}_{\chi_i}\Gamma^{i}\\
&=&q^{c_{ij}}(1+q^{c_{ii}})\Gamma^{i}_{\chi_i\chi_j}\Gamma^{i}_{\chi_j}\Gamma^{j}+\Gamma^{i}_{\chi_i\chi_j}\Gamma^{j}_{\chi_i}\Gamma^{i};
\end{eqnarray*}
\begin{eqnarray*}
\Gamma^{i}_{\chi_{j}}\Gamma^{j}\cdot \Gamma^{i}&=& [\chi_i\chi_j\cdot \Gamma^{i}][\Gamma^{i}_{\chi_j}\cdot 1][\Gamma^{j}\cdot 1]+
[\Gamma^{i}_{\chi_j}\cdot \chi_i][{\chi_j}\cdot \Gamma^{i}][ \Gamma^{j}\cdot 1]\\
&&+[ \Gamma^{i}_{\chi_j}\cdot \chi_i][\Gamma^{j}\cdot \chi_i][1\cdot \Gamma^{i}]\\
&=& \Gamma^{i}_{\chi_i\chi_j}\Gamma^{i}_{\chi_j}\Gamma^{j}+q^{c_{ii}}\Gamma^{i}_{\chi_i\chi_j}\Gamma^{i}_{\chi_j}\Gamma^{j}
+q^{c_{ii}+c_{ji}}\Gamma^{i}_{\chi_i\chi_j}\Gamma^{j}_{\chi_i}\Gamma^{i}\\
&=&(1+q^{c_{ii}})\Gamma^{i}_{\chi_i\chi_j}\Gamma^{i}_{\chi_j}\Gamma^{j}+q^{c_{ii}+c_{ji}}\Gamma^{i}_{\chi_i\chi_j}\Gamma^{j}_{\chi_i}\Gamma^{i};
\end{eqnarray*}
\begin{eqnarray*}
\Gamma^{i}\cdot\Gamma^{j}_{\chi_{i}}\Gamma^{i} &=& [\Gamma^{i}\cdot\chi_i\chi_j][1\cdot\Gamma^{i}_{\chi_j}][1\cdot\Gamma^{i}]+
[\chi_i\cdot\Gamma^{j}_{\chi_i}][\Gamma^{i}\cdot \chi_i][1\cdot \Gamma^{i}]\\
&&+[ \chi_i\cdot\Gamma^{j}_{\chi_i}][\chi_i\cdot\Gamma^{i}][\Gamma^{i}\cdot 1]\\
&=& q^{c_{ii}+c_{ij}}\Gamma^{i}_{\chi_i\chi_j}\Gamma^{j}_{\chi_i}\Gamma^{i}+q^{c_{ii}}\Gamma^{j}_{\chi^{2}_i}\Gamma^{i}_{\chi_i}\Gamma^{i}
+\Gamma^{j}_{\chi^{2}_i}\Gamma^{i}_{\chi_i}\Gamma^{i}\\
&=&(1+q^{c_{ii}})\Gamma^{j}_{\chi^{2}_i}\Gamma^{i}_{\chi_i}\Gamma^{i}+q^{c_{ii}+c_{ij}}\Gamma^{i}_{\chi_i\chi_j}\Gamma^{j}_{\chi_i}\Gamma^{i},
\end{eqnarray*}
and:
\begin{eqnarray*}
\Gamma^{j}_{\chi_{i}}\Gamma^{i}\cdot \Gamma^{i}&=& [\chi_i\chi_j\cdot \Gamma^{i}][\Gamma^{j}_{\chi_i}\cdot 1][\Gamma^{i}\cdot 1]+
[\Gamma^{j}_{\chi_i}\cdot \chi_i][{\chi_i}\cdot \Gamma^{i}][ \Gamma^{i}\cdot 1]\\
&&+[ \Gamma^{j}_{\chi_i}\cdot \chi_i][\Gamma^{i}\cdot \chi_i][1\cdot \Gamma^{i}]\\
&=& \Gamma^{i}_{\chi_i\chi_j}\Gamma^{j}_{\chi_i}\Gamma^{i}+q^{c_{ji}}\Gamma^{j}_{\chi^{2}_i}\Gamma^{i}_{\chi_i}\Gamma^{i}
+q^{c_{ji}+c_{ii}}\Gamma^{j}_{\chi^{2}_i}\Gamma^{i}_{\chi_i}\Gamma^{i}\\
&=&q^{c_{ji}}(1+q^{c_{ii}})\Gamma^{j}_{\chi^{2}_i}\Gamma^{i}_{\chi_i}\Gamma^{i}+\Gamma^{i}_{\chi_i\chi_j}\Gamma^{j}_{\chi_i}\Gamma^{i}.
\end{eqnarray*}
Here $\Gamma^{i}_{\chi_i\chi_j}\Gamma^{i}_{\chi_j}\Gamma^{j}$ is the path $1\to \chi_j\to \chi_i\chi_j\to \chi^{2}_i\chi_j$. The other paths are similar.
Thus,
\begin{eqnarray*}
\Gamma^{i}\cdot (\Gamma^{i}\cdot \Gamma^{j})&=&\Gamma^{i}\cdot (q^{c_{ij}}\Gamma^{i}_{\chi_j}\Gamma^{j}+\Gamma^{j}_{\chi_i}\Gamma^{i})\\
&=&q^{c_{ij}}(q^{c_{ij}}(1+q^{c_{ii}})\Gamma^{i}_{\chi_i\chi_j}\Gamma^{i}_{\chi_j}\Gamma^{j}+\Gamma^{i}_{\chi_i\chi_j}\Gamma^{j}_{\chi_i}\Gamma^{i}) \\
&& +((1+q^{c_{ii}})\Gamma^{j}_{\chi^{2}_i}\Gamma^{i}_{\chi_i}\Gamma^{i}+q^{c_{ii}+c_{ij}}\Gamma^{i}_{\chi_i\chi_j}\Gamma^{j}_{\chi_i}\Gamma^{i})\\
&=& q^{2c_{ij}}(1+q^{c_{ii}})\Gamma^{i}_{\chi_i\chi_j}\Gamma^{i}_{\chi_j}\Gamma^{j}+q^{c_{ij}}(1+q^{c_{ii}})\Gamma^{i}_{\chi_i\chi_j}\Gamma^{j}_{\chi_i}\Gamma^{i}
\\
&&+(1+q^{c_{ii}})\Gamma^{j}_{\chi^{2}_i}\Gamma^{i}_{\chi_i}\Gamma^{i},
\end{eqnarray*}
\begin{eqnarray*}
\Gamma^{i}\cdot (\Gamma^{j}\cdot \Gamma^{i})&=&\Gamma^{i}\cdot (q^{c_{ji}}\Gamma^{j}_{\chi_i}\Gamma^{i}+\Gamma^{i}_{\chi_j}\Gamma^{j})\\
&=&q^{c_{ji}}((1+q^{c_{ii}})\Gamma^{j}_{\chi^{2}_i}\Gamma^{i}_{\chi_i}\Gamma^{i}+q^{c_{ii}+c_{ij}}\Gamma^{i}_{\chi_i\chi_j}\Gamma^{j}_{\chi_i}\Gamma^{i})\\
&&+ (q^{c_{ij}}(1+q^{c_{ii}})\Gamma^{i}_{\chi_i\chi_j}\Gamma^{i}_{\chi_j}\Gamma^{j}+\Gamma^{i}_{\chi_i\chi_j}\Gamma^{j}_{\chi_i}\Gamma^{i})\\
&=& q^{c_{ij}}(1+q^{c_{ii}})\Gamma^{i}_{\chi_i\chi_j}\Gamma^{i}_{\chi_j}\Gamma^{j}
+(q^{c_{ii}+c_{ij}+c_{ji}}+1)\Gamma^{i}_{\chi_i\chi_j}\Gamma^{j}_{\chi_i}\Gamma^{i}\\
&&+q^{c_{ji}}(1+q^{c_{ii}})\Gamma^{j}_{\chi^{2}_i}\Gamma^{i}_{\chi_i}\Gamma^{i},
\end{eqnarray*}
and:
\begin{eqnarray*}
(\Gamma^{j}\cdot \Gamma^{i})\cdot \Gamma^{i}&=&(q^{c_{ji}}\Gamma^{j}_{\chi_i}\Gamma^{i}+\Gamma^{i}_{\chi_j}\Gamma^{j})\cdot \Gamma^{i}\\
&=&q^{c_{ji}}(q^{c_{ji}}(1+q^{c_{ii}})\Gamma^{j}_{\chi^{2}_i}\Gamma^{i}_{\chi_i}\Gamma^{i}+\Gamma^{i}_{\chi_i\chi_j}\Gamma^{j}_{\chi_i}\Gamma^{i})\\
&&+ ((1+q^{c_{ii}})\Gamma^{i}_{\chi_i\chi_j}\Gamma^{i}_{\chi_j}\Gamma^{j}+q^{c_{ii}+c_{ji}}\Gamma^{i}_{\chi_i\chi_j}\Gamma^{j}_{\chi_i}\Gamma^{i})\\
&=& (1+q^{c_{ii}})\Gamma^{i}_{\chi_i\chi_j}\Gamma^{i}_{\chi_j}\Gamma^{j}
+q^{c_{ji}}(1+q^{c_{ii}})\Gamma^{i}_{\chi_i\chi_j}\Gamma^{j}_{\chi_i}\Gamma^{i}\\
&&+q^{2c_{ji}}(1+q^{c_{ii}})\Gamma^{j}_{\chi^{2}_i}\Gamma^{i}_{\chi_i}\Gamma^{i}.
\end{eqnarray*}
If $a_{ij}=-1$, then $d_i=1, c_{ij}=-1$ and so:
$$\Gamma^{i}\cdot (\Gamma^{i}\cdot \Gamma^{j})=(1+q^{-2})\Gamma^{i}_{\chi_i\chi_j}\Gamma^{i}_{\chi_j}\Gamma^{j}+
(q+q^{-1})\Gamma^{i}_{\chi_i\chi_j}\Gamma^{j}_{\chi_i}\Gamma^{i}+(q^{2}+1)\Gamma^{j}_{\chi^{2}_i}\Gamma^{i}_{\chi_i}\Gamma^{i},$$
$$\Gamma^{i}\cdot (\Gamma^{j}\cdot \Gamma^{i})=(q+q^{-1})\Gamma^{i}_{\chi_i\chi_j}\Gamma^{i}_{\chi_j}\Gamma^{j}+
2\Gamma^{i}_{\chi_i\chi_j}\Gamma^{j}_{\chi_i}\Gamma^{i}+(q+q^{-1})\Gamma^{j}_{\chi^{2}_i}\Gamma^{i}_{\chi_i}\Gamma^{i},$$
and:
$$(\Gamma^{j}\cdot \Gamma^{i})\cdot \Gamma^{i}=(q^{2}+1)\Gamma^{i}_{\chi_i\chi_j}\Gamma^{i}_{\chi_j}\Gamma^{j}+
(q+q^{-1})\Gamma^{i}_{\chi_i\chi_j}\Gamma^{j}_{\chi_i}\Gamma^{i}+(1+q^{-2})\Gamma^{j}_{\chi^{2}_i}\Gamma^{i}_{\chi_i}\Gamma^{i}.$$

Therefore, we have the following Serre relation in case $a_{ij}=-1$:

$$\Gamma^{i}\cdot (\Gamma^{i}\cdot \Gamma^{j})-(q+q^{-1})\Gamma^{i}\cdot (\Gamma^{j}\cdot \Gamma^{i})+(\Gamma^{j}\cdot \Gamma^{i})\cdot \Gamma^{i}=0.$$
So, the proof for ADE types is done.

For the other types, we need to consider the case $a_{ij}=-2$ or $a_{ij}=-3$. Here we will omit the detail from the computation since they are similar, and
we only state the results. Let $\Gamma^{i^{r}ji^{s}}$ be the following path in $Q(M)$:
$$1\to \chi_i\to \cdots \to \chi_{i}^{s}\to \chi_{i}^{s}\chi_j\to \chi_{i}^{s+1}\chi_j\to \cdots \to \chi_{i}^{s+r}\chi_j.$$
Then we have:
\begin{eqnarray*}
(\Gamma^{i})^{3}\cdot \Gamma^{j}&=&q^{3c_{ij}}(1+q^{c_{ii}})(1+q^{c_{ii}}+q^{2c_{ii}})\Gamma^{i^{3}j}\\
&&+q^{2c_{ij}}(1+q^{c_{ii}})(1+q^{c_{ii}}+q^{2c_{ii}})\Gamma^{i^{2}ji}\\
&&+q^{c_{ij}}(1+q^{c_{ii}})(1+q^{c_{ii}}+q^{2c_{ii}})\Gamma^{iji^{2}}\\
&&+(1+q^{c_{ii}})(1+q^{c_{ii}}+q^{2c_{ii}})\Gamma^{ji^{3}};\\
(\Gamma^{i})^{2}\cdot \Gamma^{j}\cdot \Gamma^{i}&=&q^{2c_{ij}}(1+q^{c_{ii}})(1+q^{c_{ii}}+q^{2c_{ii}})\Gamma^{i^{3}j}\\
&&+q^{c_{ij}}(1+q^{c_{ii}})(1+q^{c_{ii}}+q^{2c_{ii}+2c_{ij}})\Gamma^{i^{2}ji}\\
&&+(1+q^{c_{ii}})(1+q^{c_{ii}+2c_{ij}}+q^{2c_{ii}+2c_{ij}})\Gamma^{iji^{2}}\\
&&+q^{c_{ij}}(1+q^{c_{ii}})(1+q^{c_{ii}}+q^{2c_{ii}})\Gamma^{ji^{3}};\\
\Gamma^{i}\cdot \Gamma^{j}\cdot (\Gamma^{i})^{2}&=&q^{c_{ij}}(1+q^{c_{ii}})(1+q^{c_{ii}}+q^{2c_{ii}})\Gamma^{i^{3}j}\\
&&+(1+q^{c_{ii}})(1+q^{c_{ii}+2c_{ij}}+q^{2c_{ii}+2c_{ij}})\Gamma^{i^{2}ji}\\
&&+q^{c_{ij}}(1+q^{c_{ii}})(1+q^{c_{ii}}+q^{2c_{ii}+2c_{ij}})\Gamma^{iji^{2}}\\
&&+q^{2c_{ij}}(1+q^{c_{ii}})(1+q^{c_{ii}}+q^{2c_{ii}})\Gamma^{ji^{3}};\\
\Gamma^j(\Gamma^{i})^{3}&=&(1+q^{c_{ii}})(1+q^{c_{ii}}+q^{2c_{ii}})\Gamma^{i^{3}j}\\
&&+q^{c_{ij}}(1+q^{c_{ii}})(1+q^{c_{ii}}+q^{2c_{ii}})\Gamma^{i^{2}ji}\\
&&+q^{2c_{ij}}(1+q^{c_{ii}})(1+q^{c_{ii}}+q^{2c_{ii}})\Gamma^{iji^{2}}\\
&&+q^{3c_{ij}}(1+q^{c_{ii}})(1+q^{c_{ii}}+q^{2c_{ii}})\Gamma^{ji^{3}};
\end{eqnarray*}
One can use the above results to build the Serre relation for the Lie algebras of BCF types.  We obtain:

\begin{eqnarray*}
(\Gamma^{i})^{4}\cdot \Gamma^{j}&=&q^{4c_{ij}}(1+q^{c_{ii}})(1+q^{c_{ii}}+q^{2c_{ii}})(1+q^{c_{ii}}+q^{2c_{ii}}+q^{3c_{ii}})\Gamma^{i^{4}j}\\
&&+q^{3c_{ij}}(1+q^{c_{ii}})(1+q^{c_{ii}}+q^{2c_{ii}})(1+q^{c_{ii}}+q^{2c_{ii}}+q^{3c_{ii}})\Gamma^{i^{3}ji}\\
&&+q^{2c_{ij}}(1+q^{c_{ii}})(1+q^{c_{ii}}+q^{2c_{ii}})(1+q^{c_{ii}}+q^{2c_{ii}}+q^{3c_{ii}})\Gamma^{i^{2}ji^{2}}\\
&&+q^{c_{ij}}(1+q^{c_{ii}})(1+q^{c_{ii}}+q^{2c_{ii}})(1+q^{c_{ii}}+q^{2c_{ii}}+q^{3c_{ii}})\Gamma^{iji^{3}}\\
&&+(1+q^{c_{ii}})(1+q^{c_{ii}}+q^{2c_{ii}})(1+q^{c_{ii}}+q^{2c_{ii}}+q^{3c_{ii}})\Gamma^{ji^{4}};
\end{eqnarray*}

\begin{eqnarray*}
(\Gamma^{i})^{3}\cdot \Gamma^{j}\cdot \Gamma^{i}&=&q^{3c_{ij}}(1+q^{c_{ii}})(1+q^{c_{ii}}+q^{2c_{ii}})(1+q^{c_{ii}}+q^{2c_{ii}}+q^{3c_{ii}})\Gamma^{i^{4}j}\\
&&+q^{2c_{ij}}(1+q^{c_{ii}})(1+q^{c_{ii}}+q^{2c_{ii}})(1+q^{c_{ii}}+q^{2c_{ii}}+q^{3c_{ii}+2c_{ij}})\Gamma^{i^{3}ji}\\
&&+q^{c_{ij}}(1+q^{c_{ii}})(1+q^{c_{ii}}+q^{2c_{ii}})(1+q^{c_{ii}}+q^{2c_{ii}+2c_{ij}}+q^{3c_{ii}+2c_{ij}})\Gamma^{i^{2}ji^{2}}\\
&&+(1+q^{c_{ii}})(1+q^{c_{ii}}+q^{2c_{ii}})(1+q^{c_{ii}+2c_{ij}}+q^{2c_{ii}+2c_{ij}}+q^{3c_{ii}+2c_{ij}})\Gamma^{iji^{3}}\\
&&+q^{c_{ij}}(1+q^{c_{ii}})(1+q^{c_{ii}}+q^{2c_{ii}})(1+q^{c_{ii}}+q^{2c_{ii}}+q^{3c_{ii}})\Gamma^{ji^{4}};
\end{eqnarray*}
\begin{eqnarray*}
(\Gamma^{i})^{2}\cdot \Gamma^{j}\cdot (\Gamma^{i})^{2}&=&q^{2c_{ij}}(1+q^{c_{ii}})(1+q^{c_{ii}}+q^{2c_{ii}})(1+q^{c_{ii}}+q^{2c_{ii}}+q^{3c_{ii}})\Gamma^{i^{4}j}\\
&&+q^{c_{ij}}(1+q^{c_{ii}})(1+q^{c_{ii}}+q^{2c_{ii}})(1+q^{c_{ii}}+q^{2c_{ii}+2c_{ij}}+q^{3c_{ii}+2c_{ij}})\Gamma^{i^{3}ji}\\
&&+(1+q^{c_{ii}})^{2}(1+q^{c_{ii}+2c_{ij}}+2q^{2c_{ii}+2c_{ij}}+q^{3c_{ii}+2c_{ij}}+q^{4c_{ii}+4c_{ij}})\Gamma^{i^{2}ji^{2}}\\
&&+q^{c_{ij}}(1+q^{c_{ii}})(1+q^{c_{ii}}+q^{2c_{ii}})(1+q^{c_{ii}}+q^{2c_{ii}+2c_{ij}}+q^{3c_{ii}+2c_{ij}})\Gamma^{iji^{3}}\\
&&+q^{2c_{ij}}(1+q^{c_{ii}})(1+q^{c_{ii}}+q^{2c_{ii}})(1+q^{c_{ii}}+q^{2c_{ii}}+q^{3c_{ii}})\Gamma^{ji^{4}};
\end{eqnarray*}
\begin{eqnarray*}
\Gamma^{i}\cdot \Gamma^{j}\cdot (\Gamma^i)^3&=&q^{c_{ij}}(1+q^{c_{ii}})(1+q^{c_{ii}}+q^{2c_{ii}})(1+q^{c_{ii}}+q^{2c_{ii}}+q^{3c_{ii}})\Gamma^{i^{4}j}\\
&&+(1+q^{c_{ii}})(1+q^{c_{ii}}+q^{2c_{ii}})(1+q^{c_{ii}+2c_{ij}}+q^{2c_{ii}+2c_{ij}}+q^{3c_{ii}+2c_{ij}})\Gamma^{i^{3}ji}\\
&&+q^{c_{ij}}(1+q^{c_{ii}})(1+q^{c_{ii}}+q^{2c_{ii}})(1+q^{c_{ii}}+q^{2c_{ii}+2c_{ij}}+q^{3c_{ii}+2c_{ij}})\Gamma^{i^{2}ji^{2}}\\
&&+q^{2c_{ij}}(1+q^{c_{ii}})(1+q^{c_{ii}}+q^{2c_{ii}})(1+q^{c_{ii}}+q^{2c_{ii}}+q^{3c_{ii}+2c_{ij}})\Gamma^{iji^{3}}\\
&&+q^{3c_{ij}}(1+q^{c_{ii}})(1+q^{c_{ii}}+q^{2c_{ii}})(1+q^{c_{ii}}+q^{2c_{ii}}+q^{3c_{ii}})\Gamma^{ji^{4}};
\end{eqnarray*}
\begin{eqnarray*}
\Gamma^{j}\cdot(\Gamma^{i})^{4}&=&(1+q^{c_{ii}})(1+q^{c_{ii}}+q^{2c_{ii}})(1+q^{c_{ii}}+q^{2c_{ii}}+q^{3c_{ii}})\Gamma^{i^{4}j}\\
&&+q^{c_{ij}}(1+q^{c_{ii}})(1+q^{c_{ii}}+q^{2c_{ii}})(1+q^{c_{ii}}+q^{2c_{ii}}+q^{3c_{ii}})\Gamma^{i^{3}ji}\\
&&+q^{2c_{ij}}(1+q^{c_{ii}})(1+q^{c_{ii}}+q^{2c_{ii}})(1+q^{c_{ii}}+q^{2c_{ii}}+q^{3c_{ii}})\Gamma^{i^{2}ji^{2}}\\
&&+q^{3c_{ij}}(1+q^{c_{ii}})(1+q^{c_{ii}}+q^{2c_{ii}})(1+q^{c_{ii}}+q^{2c_{ii}}+q^{3c_{ii}})\Gamma^{iji^{3}}\\
&&+q^{4c_{ij}}(1+q^{c_{ii}})(1+q^{c_{ii}}+q^{2c_{ii}})(1+q^{c_{ii}}+q^{2c_{ii}}+q^{3c_{ii}})\Gamma^{ji^{4}}.
\end{eqnarray*}
Similarly, the Serre relation for type $G_{2}$ can be obtained from the above equalities.
\end{proof}

\section{The Drinfeld double of $A_{q}(\mathfrak{g})$}
In this section, we will determine the structure of $D(A_{q}(\mathfrak{g}))$. We shall first  describe the  algebraic structure.  To this end, we need to compute the elements $\gamma,\mathbf{f},\chi$ and $\omega$ according to
the formulas \eqref{eq;1.14}-\eqref{eq;1.15}.
The following easy observation \cite[Lem. 3.2]{Liu} will be used frequently throughout the paper.

\begin{lemma}\label{l4.1} For any two natural numbers $i,j$, we have following identity:
\begin{equation}[\frac{i+j'}{n}]=[\frac{i+j}{n}]-[\frac{j}{n}].
\end{equation}\end{lemma}

\begin{lemma}\label{14.2} For the quasi-Hopf algebra $A_{q}(\mathfrak{g})$, we have:
\begin{gather*} \gamma=\sum_{b,c\in (\mathbbm{Z}_{n})^{m}}\prod_{i,j=1}^{m}
\q^{c_{ij}(-(b_{i}+c_{i})[\frac{b_{j}+c_{j}}{n}]+c_{i}-
c_{i}[\frac{n-b_{j}}{n}]+b_{i}[\frac{n-1+b_{j}}{n}]+c_{i}[\frac{n-1+c_{j}}{n}])}\mathbf{1}_{b}\otimes \mathbf{1}_{c},\\
\mathbf{f}=\sum_{e,f\in (\mathbbm{Z}_{n})^{m}}\prod_{i,j=1}^{m}
\q^{c_{ij}(-e_{i}+(e_{i}+f_{i})([\frac{n-(e_{j}+f_{j})'}{n}]-[\frac{e_{j}+f_{j}}{n}])-f_{i}[\frac{n-e_{j}}{n}]
+e_{i}[\frac{n-1+e_{j}}{n}]+f_{i}[\frac{n-1+f_{j}}{n}])}\\ \times \mathbf{1}_{e}\otimes \mathbf{1}_{f},\\
\chi=\sum_{a,b,c,d\in (\mathbbm{Z}_{n})^{m}}\prod_{i,j=1}^{m}
\q^{c_{ij}(-a_{i}[\frac{b_{j}+c_{j}}{n}]+(a_{i}+b_{i})[\frac{c_{i}+d_{i}}{n}])}\mathbf{1}_{a}\otimes \mathbf{1}_{b}
\otimes \mathbf{1}_{c}\otimes \mathbf{1}_{d},\end{gather*}
and:
\begin{eqnarray*}\omega&=&\sum_{a,b,c,d,e\in (\mathbbm{Z}_{n})^{m}}\prod_{i,j=1}^{m}
\q^{c_{ij}(-a_{i}[\frac{b_{j}+c_{j}+d_{j}}{n}]+(a_{i}+b_{i}+c_{i}+d_{i}+e_{i})[\frac{d_{j}+e_{j}}{n}]
-e_{i}+e_{i}[\frac{n-d_j}{n}])}\\
&&\;\;\;\;\;\;\;\;\;\;\;\;\;\;\;\;\;\;\;\;\;\;\;\times\mathbf{1}_{a}\otimes \mathbf{1}_{b}
\otimes \mathbf{1}_{c}\otimes S(\mathbf{1}_{d}) \otimes S(\mathbf{1}_{e}).
\end{eqnarray*}
\end{lemma}
\begin{proof} By definition, we have
\begin{eqnarray*}
&& T^{i}\otimes U^{i}\otimes V^{i}\otimes W^{i}\\
&=&(1\otimes \phi^{-1})(id\otimes id\otimes\D)(\phi)\\
&=&\sum_{a,b,c\in (\mathbbm{Z}_{n})^{m}}\prod_{i,j=1}^{m}
\q^{c_{ij}a_{i}[\frac{b_{j}+c_{j}}{n}]} 1\otimes\mathbf{1}_{a}\otimes \mathbf{1}_{b}
\otimes \mathbf{1}_{c}\times\\
&& \sum_{d,e,f^{1},f^{2}\in (\mathbbm{Z}_{n})^{m}}\prod_{i,j=1}^{m}
\q^{-c_{ij}d_{i}[\frac{e_{j}+(f^{1}_{j}+f^{2}_{j})'}{n}]} \mathbf{1}_{d}\otimes\mathbf{1}_{e}\otimes \mathbf{1}_{f^1}
\otimes \mathbf{1}_{f^2}\\
&=&\sum_{a,b,c,d\in (\mathbbm{Z}_{n})^{m}}\prod_{i,j=1}^{m}\q^{c_{ij}(a_{i}[\frac{b_{j}+c_{j}}{n}]-d_{i}[\frac{a_{j}+(b_{j}+c_{j})'}{n}])}\mathbf{1}_{d}\otimes \mathbf{1}_{a}
\otimes \mathbf{1}_{b}\otimes \mathbf{1}_{c}
\end{eqnarray*}
and so:
\begin{eqnarray*}\gamma&=&(S(U^{i})\otimes S(T^{i}))(\alpha\otimes\alpha)(V^{i}\otimes W^{i})\\
&=&\sum_{a,b,c,d\in (\mathbbm{Z}_{n})^{m}}\prod_{i,j=1}^{m}\q^{c_{ij}(a_{i}[\frac{b_{j}+c_{j}}{n}]-d_{i}[\frac{a_{j}+(b_{j}+c_{j})'}{n}])}
S(\mathbf{1}_{a})\alpha\mathbf{1}_{b}\otimes S(\mathbf{1}_{d})\alpha\mathbf{1}_{c}\\
&=&\sum_{a,b,c,d\in (\mathbbm{Z}_{n})^{m}}\prod_{i,j=1}^{m}\q^{c_{ij}(a_{i}[\frac{b_{j}+c_{j}}{n}]-d_{i}[\frac{a_{j}+(b_{j}+c_{j})'}{n}]+
b_{i}[\frac{n-1+b_{j}}{n}]+c_{i}[\frac{n-1+c_{j}}{n}])}\\
&&\;\;\;\;\;\;\;\;\;\;\;\;\;\;\;\;\;\;\;\;\times S(\mathbf{1}_{a})\mathbf{1}_{b}\otimes S(\mathbf{1}_{d})\mathbf{1}_{c}\\
&=&\sum_{b,c\in (\mathbbm{Z}_{n})^{m}}\prod_{i,j=1}^{m}\q^{c_{ij}((n-b_{i})'[\frac{b_{j}+c_{j}}{n}]-(n-c_{i})'[\frac{(n-b_{j})'+(b_{j}+c_{j})'}{n}]+
b_{i}[\frac{n-1+b_{j}}{n}]+c_{i}[\frac{n-1+c_{j}}{n}])}\\
&&\;\;\;\;\;\;\;\;\;\;\;\;\;\;\;\times \mathbf{1}_{b}\otimes \mathbf{1}_{c}
\end{eqnarray*}
\begin{eqnarray*}
&=&\sum_{b,c\in (\mathbbm{Z}_{n})^{m}}\prod_{i,j=1}^{m}\q^{c_{ij}(-b_{i}[\frac{b_{j}+c_{j}}{n}]+c_{i}([\frac{n-b_{j}+b_{j}+c_{j}}{n}]
-[\frac{n-b_{j}}{n}]-[\frac{b_{j}+c_{j}}{n}])+
b_{i}[\frac{n-1+b_{j}}{n}]+c_{i}[\frac{n-1+c_{j}}{n}])}\\&&\;\;\;\;\;\;\;\;\;\;\;\;\;\;\;\times \mathbf{1}_{b}\otimes \mathbf{1}_{c}\\
&=&\sum_{b,c\in (\mathbbm{Z}_{n})^{m}}\prod_{i,j=1}^{m}\q^{c_{ij}(-b_{i}[\frac{b_{j}+c_{j}}{n}]+c_{i}(1
-[\frac{n-b_{j}}{n}]-[\frac{b_{j}+c_{j}}{n}])+
b_{i}[\frac{n-1+b_{j}}{n}]+c_{i}[\frac{n-1+c_{j}}{n}])}\\
&&\;\;\;\;\;\;\;\;\;\;\;\;\;\;\;\times \mathbf{1}_{b}\otimes \mathbf{1}_{c}\\
&=&\sum_{b,c\in (\mathbbm{Z}_{n})^{m}}\prod_{i,j=1}^{m}\q^{c_{ij}(-(b_{i}+c_{i})
[\frac{b_{j}+c_{j}}{n}]+c_{i}-c_{i}[\frac{n-b_{j}}{n}]+b_{i}[\frac{n-1+b_{j}}{n}]+c_{i}[\frac{n-1+c_{j}}{n}])}\\
&&\;\;\;\;\;\;\;\;\;\;\;\;\;\;\;\times \mathbf{1}_{b}\otimes \mathbf{1}_{c},
\end{eqnarray*}
where the fifth equality follows from Lemma \ref{l4.1}.

Therefore:
\begin{eqnarray*}\mathbf{f}&=&\sum (S\otimes S)(\D^{op}(\overline{X}^{i}))\cdot \mathbf{\gamma} \cdot \D(\overline{Y}^{i}\beta S(\overline{Z}^{i}))\\
&=&\sum_{a^{1},a^{2},b,c\in (\mathbbm{Z}_{n})^{m}}\prod_{i,j=1}^{m}\q^{c_{ij}(a^{1}_{i}+a^{2}_{i})[\frac{b_{j}+c_{j}}{n}]}
(S(\mathbf{1}_{a^{1}})\otimes S(\mathbf{1}_{a^{2}}))\gamma\D(\mathbf{1}_{b}S( \mathbf{1}_{c}))\\
&=&\sum_{e,f\in (\mathbbm{Z}_{n})^{m}}\prod_{i,j=1}^{m}\q^{c_{ij}((n-e_{i})'+(n-f_{i})')[\frac{(e_{j}+f_{j})'+(n-(e_{j}+f_{j})')'}{n}]}
\gamma \mathbf{1}_{e}\otimes \mathbf{1}_{f}\\
&=&\sum_{e,f\in (\mathbbm{Z}_{n})^{m}}\prod_{i,j=1}^{m}\q^{-c_{ij}(e_{i}+f_{i})(1-[\frac{n-(e_{j}+f_{j})'}{n}])}
\gamma \mathbf{1}_{e}\otimes \mathbf{1}_{f}\\
&=&\sum_{e,f\in (\mathbbm{Z}_{n})^{m}}\prod_{i,j=1}^{m}\q^{-c_{ij}(e_{i}+f_{i})(1-[\frac{n-(e_{j}+f_{j})'}{n}])}\\
&&\q^{c_{ij}(-(e_{i}+f_{i})
[\frac{e_{j}+f_{j}}{n}]+f_{i}-f_{i}[\frac{n-e_{j}}{n}]+e_{i}[\frac{n-1+e_{j}}{n}]+f_{i}[\frac{n-1+f_{j}}{n}])}\mathbf{1}_{e}\otimes \mathbf{1}_{f} \\
&=&\sum_{e,f\in (\mathbbm{Z}_{n})^{m}}\prod_{i,j=1}^{m}
\q^{c_{ij}(-e_{i}+(e_{i}+f_{i})([\frac{n-(e_{j}+f_{j})'}{n}]-[\frac{e_{j}+f_{j}}{n}])-f_{i}[\frac{n-e_{j}}{n}]
+e_{i}[\frac{n-1+e_{j}}{n}]+f_{i}[\frac{n-1+f_{j}}{n}])}\\
&&\;\;\;\;\;\;\;\;\;\;\;\;\;\;\;\; \times\mathbf{1}_{e}\otimes \mathbf{1}_{f},
\end{eqnarray*}
where the fourth equality follows also from Lemma \ref{l4.1}.

The computation for $\chi$ is easy. Indeed:
\begin{eqnarray*}\chi&=&(\phi\otimes 1)(\D\otimes id\otimes id)(\phi^{-1})\\
&=&\sum_{a,b,c\in (\mathbbm{Z}_{n})^{m}}\prod_{i,j=1}^{m}\q^{-c_{ij}a_{i}[\frac{b_{j}+c_{j}}{n}]}\mathbf{1}_{a}\otimes \mathbf{1}_{b}
\otimes \mathbf{1}_{c} \otimes 1\\
&&\sum_{d^{1},d^{2},e,f\in (\mathbbm{Z}_{n})^{m}}\prod_{i,j=1}^{m}\q^{c_{ij}(d^{1}_{i}+d^{2}_{i})[\frac{e_{j}+f_{j}}{n}]}\mathbf{1}_{d^{1}}\otimes \mathbf{1}_{d^{2}}
\otimes \mathbf{1}_{e} \otimes \mathbf{1}_{f}\\
&=&\sum_{a,b,c,d\in (\mathbbm{Z}_{n})^{m}}\prod_{i,j=1}^{m}
\q^{c_{ij}(-a_{i}[\frac{b_{j}+c_{j}}{n}]+(a_{i}+b_{i})[\frac{c_{j}+d_{j}}{n}])}\mathbf{1}_{a}\otimes \mathbf{1}_{b}
\otimes \mathbf{1}_{c}\otimes \mathbf{1}_{d}.
\end{eqnarray*}
Now we are able to compute the element $\omega$.
\begin{eqnarray*} \omega&=&(1\otimes 1\otimes 1\otimes \tau(\mathbf{f}^{-1}))(id\otimes \D\otimes S\otimes S)(\mathbf{\chi})(\phi\otimes 1\otimes 1)\\
&=& \sum_{e,f\in (\mathbbm{Z}_{n})^{m}}\prod_{i,j=1}^{m}
\q^{c_{ij}(e_{i}-(e_{i}+f_{i})([\frac{n-(e_{j}+f_{j})'}{n}]-[\frac{e_{j}+f_{j}}{n}])+f_{i}[\frac{n-e_{j}}{n}]
-e_{i}[\frac{n-1+e_{j}}{n}]-f_{i}[\frac{n-1+f_{j}}{n}])}\\
&&\;\;\;\;\;\;\;\;\;\;\;\;\;\;\;\; \ \  1\otimes 1\otimes 1\otimes \mathbf{1}_{f}\otimes \mathbf{1}_{e}\\
&&\sum_{a,b^{1},b^{2},c,d\in (\mathbbm{Z}_{n})^{m}}\prod_{i,j=1}^{m}
\q^{c_{ij}(-a_{i}[\frac{(b^1_{j}+b^2_{j})'+c_{j}}{n}]+(a_{i}+b^1_{i}+b^2_{i})[\frac{c_{j}+d_{j}}{n}])}\mathbf{1}_{a}\otimes \mathbf{1}_{b^{1}}
\otimes \mathbf{1}_{b^{2}}
\otimes S(\mathbf{1}_{c})\otimes S(\mathbf{1}_{d})\\
&&\sum_{a,b^{1},b^{2}\in (\mathbbm{Z}_{n})^{m}}\prod_{i,j=1}^{m}
\q^{-c_{ij}a_{i}[\frac{b^1_{j}+b^2_{j}}{n}]} \mathbf{1}_{a}\otimes \mathbf{1}_{b^{1}}
\otimes \mathbf{1}_{b^{2}}
\otimes 1\otimes 1\\
&=&\sum_{a,b^{1},b^{2},c,d\in (\mathbbm{Z}_{n})^{m}}\prod_{i,j=1}^{m}
\q^{c_{ij}(-a_{i}[\frac{b^1_{j}+b^2_{j}+c_{j}}{n}]+(a_{i}+b^1_{i}+b^2_{i})[\frac{c_{j}+d_{j}}{n}]-d_{i}-c_{i}[\frac{n-(n-d_{j})'}{n}])}\\
&&\q^{c_{ij}((d_{i}+c_{i})
([\frac{n-((n-d_{j})'+(n-c_{j})')'}{n}]-[\frac{(n-d_{j})'+(n-c_{j})'}{n}])
+d_{i}[\frac{n-1+(n-d_{j})'}{n}]+c_{i}[\frac{n-1+(n-c_{j})'}{n}])}\\
&&\mathbf{1}_{a}\otimes \mathbf{1}_{b^{1}}
\otimes \mathbf{1}_{b^{2}}
\otimes S(\mathbf{1}_{c})\otimes S(\mathbf{1}_{d}).
\end{eqnarray*}
Note that we used Lemma \ref{l4.1} in the computation in order to obtain:
 $$\q^{-c_{ij}a_{i}[\frac{b^1_{j}+b^2_{j}}{n}]}\q^{c_{ij}
(-a_{i}[\frac{(b^1_{j}+b^2_{j})'+c_{j}}{n}])}=
\q^{-c_{ij}a_{i}[\frac{b^1_{j}+b^2_{j}+c_{j}}{n}]}.$$
The following equalities can be verified directly:
$$\begin{array}{rll}
[\frac{n-((n-d_{j})'+(n-c_{j})')'}{n}] &=& [\frac{n-(d_{j}+c_{j})'}{n}],\\

[\frac{(n-d_{j})'+(n-c_{j})'}{n}]&=&[\frac{2n-d_{j}-c_{j}}{n}]-[\frac{n-d_{j}}{n}]-[\frac{n-c_{j}}{n}]\\
&=& 1+[\frac{n-(d_{j}+c_{j})'}{n}]-[\frac{d_{j}+c_{j}}{n}]-[\frac{n-d_{j}}{n}]-[\frac{n-c_{j}}{n}],\\

[\frac{n-(n-d_{j})'}{n}]&=&[\frac{n-d_{j}}{n}],\\

[\frac{n-1+(n-d_{j})'}{n}] &=& [\frac{n-1+d_{j}}{n}],\\

[\frac{n-1+(n-c_{j})'}{n}] &=& [\frac{n-1+c_{j}}{n}].
\end{array}$$

Applying these equations to the expression of $\omega$, we obtain:
\begin{eqnarray*}
\omega &=&\sum_{a,b^{1},b^{2},c,d\in (\mathbbm{Z}_{n})^{m}}\prod_{i,j=1}^{m}
\q^{c_{ij}(-a_{i}[\frac{b^1_{j}+b^2_{j}+c_{j}}{n}]+(a_{i}+b^1_{i}+b^2_{i})[\frac{c_{j}+d_{j}}{n}]-d_{i}-c_{i}[\frac{n-d_{j}}{n}])}\\
&&\q^{c_{ij}((d_{i}+c_{i})
(-1+[\frac{d_{j}+c_{j}}{n}]+[\frac{n-d_{j}}{n}]+[\frac{n-c_{j}}{n}])
+d_{i}[\frac{n-1+d_{j}}{n}]+c_{i}[\frac{n-1+c_{j}}{n}])}\\
&&\mathbf{1}_{a}\otimes \mathbf{1}_{b^{1}}
\otimes \mathbf{1}_{b^{2}}
\otimes S(\mathbf{1}_{c})\otimes S(\mathbf{1}_{d}) \\
&=& \sum_{a,b^{1},b^{2},c,d\in (\mathbbm{Z}_{n})^{m}}\prod_{i,j=1}^{m}
\q^{c_{ij}(-a_{i}[\frac{b^1_{j}+b^2_{j}+c_{j}}{n}]+(a_{i}+b^1_{i}+b^2_{i}+d_{i}+c_{i})[\frac{c_{j}+d_{j}}{n}]
)}\\
&&\q^{c_{ij}(d_{i}(-1-1+[\frac{n-d_{j}}{n}]+[\frac{n-c_{j}}{n}]+[\frac{n-1+d_{j}}{n}])+
c_{i}(-1+[\frac{n-d_{j}}{n}]+[\frac{n-c_{j}}{n}]-[\frac{n-d_{j}}{n}]+[\frac{n-1+c_{j}}{n}]))}\\
&&\mathbf{1}_{a}\otimes \mathbf{1}_{b^{1}}
\otimes \mathbf{1}_{b^{2}}
\otimes S(\mathbf{1}_{c})\otimes S(\mathbf{1}_{d}).
\end{eqnarray*}
Now we can apply the following identity to simplify the formula of $\omega$:
$$[\frac{n-1+z}{n}]+[\frac{n-z}{n}]=1,\;\;\;\;\textrm{for}\;0\leq z< n.$$
It follows that

\begin{eqnarray*}
\omega &=& \sum_{a,b^{1},b^{2},c,d\in (\mathbbm{Z}_{n})^{m}}\prod_{i,j=1}^{m}
\q^{c_{ij}(-a_{i}[\frac{b^1_{j}+b^2_{j}+c_{j}}{n}]+(a_{i}+b^1_{i}+b^2_{i}+d_{i}+c_{i})[\frac{c_{j}+d_{j}}{n}]
-d_{i}+d_{i}[\frac{n-c_{j}}{n}])}\\
&&\mathbf{1}_{a}\otimes \mathbf{1}_{b^{1}}
\otimes \mathbf{1}_{b^{2}}
\otimes S(\mathbf{1}_{c})\otimes S(\mathbf{1}_{d})\\
&=&\sum_{a,b,c,d,e\in (\mathbbm{Z}_{n})^{m}}\prod_{i,j=1}^{m}
\q^{c_{ij}(-a_{i}[\frac{b_{j}+c_{j}+d_{j}}{n}]+(a_{i}+b_{i}+c_{i}+d_{i}+e_{i})[\frac{d_{j}+e_{j}}{n}]
-e_{i}+e_{i}[\frac{n-d_j}{n}])}\\
&&\;\;\;\;\;\;\;\;\;\;\;\;\;\;\;\;\;\;\;\  \ \mathbf{1}_{a}\otimes \mathbf{1}_{b}
\otimes \mathbf{1}_{c}\otimes S(\mathbf{1}_{d}) \otimes S(\mathbf{1}_{e}).
\end{eqnarray*}
\end{proof}

The algebraic structure of $D(A_{q}(\mathfrak{g}))$ can be described by the following three propositions, which can be
understood roughly as ``the generating relations for $A_q(\mathfrak{g})$", ``the generating relations for  $M_q(\mathfrak{g})$" and ``the generating relations between $A_q(\mathfrak{g})$
and $M_{q}(\mathfrak{g})$" respectively.

\begin{proposition}\label{p4.3} In $D(A_{q}(\mathfrak{g}))$, we have the following relations
\begin{eqnarray}&&(h_{i}\bi \e)^{n}=1\bi \e,\;\;(h_{i}\bi \e)(h_{j}\bi \e)=(h_{j}\bi\e)(h_{i}\bi\e),\\
&&  (h_{i}\bi\e)(e_{j}\bi\e)(h_{i}\bi\e)^{-1}=\q^{\delta_{i,j}}(e_{j}\bi \e),\;\;(e_{i}\bi\e)^{l_{i}}=0,\\
 && \sum_{r+s=1-a_{ij}}(-1)^{s}\left [
\begin{array}{c} 1-a_{ij}\\s
\end{array}\right]_{d_{i}}(e_{i}\bi\e)^{r}(e_{j}\bi\e)(e_{i}\bi\e)^{s}=0,\;\;\textrm{if}\;\;i\neq j.
\end{eqnarray}
for $1\leq i,j\leq m$.
\end{proposition}
\begin{proof} Follows the fact that $A_{q}(\mathfrak{g})$ is a quasi-Hopf subalgebra of $D(A_{q}(\mathfrak{g}))$.
  \end{proof}

As a subspace, we always have a natural embedding $M_{q}(\mathfrak{g})\hookrightarrow D(A_{q}(\mathfrak{g}))$ through
$\varphi\mapsto 1\bowtie \varphi$ for $\varphi\in M_{q}(\mathfrak{g})$.

\begin{lemma}\label{l4.2} For $h\in A_{q}(\mathfrak{g})$ and $\varphi \in M_{q}(\mathfrak{g})$, we have
$$(h\bowtie \varepsilon)(1\bowtie \varphi)=h\bowtie \varphi.$$
\end{lemma}
\begin{proof} This is  a special case of Remark 6.2 in \cite{Sch}.
The reader also can prove it by using Formula $(\star)$ in Theorem \ref{t1}.
\end{proof}

Following Lemma \ref{l4.2}, we have no worry to write $h\varphi$  for $h\bowtie \varphi$  and  just denote by $h$ the element $h\bowtie \varepsilon$ for short.  Recall that we have already defined the elements $\flat_{i}$ and $H_i$ in (\ref{eq;1}). These elements will be used in the following propositions.

\begin{proposition}\label{p4.5} Assume $n\geq 4$. Then we have the following relations in $D(A_{q}(\mathfrak{g}))$:
 \begin{eqnarray}&&(1\bi \chi_{i})(1\bi \chi_j)=(1\bi \chi_j)(1\bi \chi_i),\;\;\;(\flat_i\Gamma^i)^{l_{i}}=0,\\
&& \sum_{r+s=1-a_{ij}}(-1)^{s}\left [
\begin{array}{c} 1-a_{ij}\\s
\end{array}\right]_{d_{i}}(\flat_i\Gamma^i)^{r} (\flat_j\Gamma^{j}) (\flat_i\Gamma^{i})^{s}=0,
\end{eqnarray}
for $1\leq i\neq j\leq m$ and $l_{i}=\ord(q^{c_{ii}})$.
\end{proposition}
\begin{proof} For the first part of (4.5), we have:
\begin{eqnarray*}
(1\bi \chi_{i})(1\bi \chi_j)&=&\omega^{(3)}\bi(\omega^{(5)}\rightharpoonup \chi_j\leftharpoonup \omega^{(1)})(\omega^{(4)}\rightharpoonup \chi_i\leftharpoonup \omega^{(2)})\\
&=& \sum_{a=\epsilon_{j},b=\epsilon_{i},c\in (\mathbbm{Z}_{n})^{m},d=(n-1)\epsilon_{i},e=(n-1)\epsilon_{j}}\prod_{s\neq j,t}\q^{c_{st}0}
\\
&&\prod_{s= j,t\neq i}\q^{c_{jt}0}\q^{c_{ji}(-[\frac{n+c_{i}}{n}]-(n-1))}\mathbf{1}_{c}\bi (\chi_i\cdot \chi_j)\\
&=&1\bi (\chi_i\cdot\chi_j).
\end{eqnarray*}
Similarly, one can show that $(1\bi \chi_{j})(1\bi \chi_i)=1\bi (\chi_j\cdot\chi_i)$ and so $(1\bi \chi_{i})(1\bi \chi_j)=(1\bi \chi_{j})(1\bi \chi_i)$.
Applying the proof of Formula (3.8) in \cite[Prop. 3.4]{Liu}, we get $(1\bi \Gamma^{i})^{l_{i}}=0$. By (4.8) of the next Proposition \ref{p4.6}, we have:
$$\flat_{i}\Gamma^{i}\mathbf{1}_{a}=\mathbf{1}_{a-\epsilon_{i}}\flat_{i}\Gamma^{i}\;\;\;\;(\textrm{and so}, \;\;\Gamma^{i}\mathbf{1}_{a}=\mathbf{1}_{a-\epsilon_{i}}\Gamma^{i}).$$
Therefore, $(\flat_{i}\Gamma^{i})^{l_{i}}=0$.

Finally, we show the Serre relation.
\begin{eqnarray*}
(\flat_{i}\Gamma^{i})(\flat_{j}\Gamma^{j})&=&\flat_{i}\sum_{a\in(\mathbbm{Z}_{n})^{m}}\prod_{l=1}^{m}q^{-c_{jl}a_{l}}\mathbf{1}_{a-\epsilon_{i}}
(1\bi \Gamma^i)(1\bi \Gamma^j)\\
&=&\flat_{i}\sum_{a\in(\mathbbm{Z}_{n})^{m}}\prod_{l\neq i}q^{-c_{jl}a_{l}}q^{-c_{ji}(1+a_{i})'}\mathbf{1}_{a}
\\
&&\omega^{(3)}\bi(\omega^{(5)}\rightharpoonup \Gamma^j\leftharpoonup \omega^{(1)})(\omega^{(4)}\rightharpoonup \Gamma^i\leftharpoonup \omega^{(2)})\\
&=&\flat_{i}\sum_{a\in(\mathbbm{Z}_{n})^{m}}\prod_{l\neq i}q^{-c_{jl}a_{l}}q^{-c_{ji}(1+a_{i})'}\mathbf{1}_{a}\\
&&\sum_{a=\epsilon_{j},b=\epsilon_{i},c\in (\mathbbm{Z}_{n})^{m},d=e=0\epsilon_{i}}\q^{-c_{ji}[\frac{1+c_{i}}{n}]}\mathbf{1}_{c}\bi (\Gamma^j\cdot \Gamma^i)\\
&=&\flat_{i}\sum_{a\in(\mathbbm{Z}_{n})^{m}}\prod_{l\neq i}q^{-c_{jl}a_{l}}q^{-c_{ji}(1+a_{i})}\mathbf{1}_{a}\bi (\Gamma^j\cdot \Gamma^i)\\
&=&q^{-c_{ji}}\flat_{i}\flat_{j}(\Gamma^j\cdot \Gamma^i).
\end{eqnarray*}
Similarly, we have $(\flat_{j}\Gamma^{j})(\flat_{i}\Gamma^{i})=q^{-c_{ij}}\flat_{i}\flat_{j}(\Gamma^i\cdot \Gamma^j)$,
and
\begin{eqnarray*}
(\flat_{i}\Gamma^{i})^{2}(\flat_{j}\Gamma^{j})&=&\flat_{i}\sum_{a\in(\mathbbm{Z}_{n})^{m}}\prod_{l\neq i}q^{-c_{il}a_{l}}q^{-c_{ii}(1+a_{i})'}\mathbf{1}_{a}\sum_{b\in(\mathbbm{Z}_{n})^{m}}\prod_{l\neq i}q^{-c_{jl}b_{l}}q^{-c_{ji}(2+b_{i})'}\mathbf{1}_{b}\\
&&(1\bi \Gamma^i)^2(1\bi \Gamma^j)
\\
&=&\flat_{i}\sum_{a\in(\mathbbm{Z}_{n})^{m}}\prod_{l\neq i}q^{-c_{il}a_{l}}q^{-c_{ii}(1+a_{i})'}\mathbf{1}_{a}\sum_{b\in(\mathbbm{Z}_{n})^{m}}\prod_{l\neq i}q^{-c_{jl}b_{l}}q^{-c_{ji}(2+b_{i})'}\mathbf{1}_{b}\\
&& \sum_{c\in(\mathbbm{Z}_{n})^{m}}\q^{-c_{ii}[\frac{1+c_{i}}{n}]-c_{ji}[\frac{2+c_{i}}{n}]}\mathbf{1}_{c}\bi \Gamma^j\cdot(\Gamma^i\cdot\Gamma^i)\\
&=&q^{-2c_{ji}-c_{ii}}\flat_{i}^{2}\flat_{j} \Gamma^j\cdot(\Gamma^i\cdot\Gamma^i).
\end{eqnarray*}
In a similar way, we obtain the following identities:
$$(\flat_{i}\Gamma^{i})(\flat_{j}\Gamma^{j})(\flat_{i}\Gamma^{i})=q^{-c_{ii}-c_{ij}-c_{ji}}\flat_{i}^{2}\flat_{j} \Gamma^i\cdot \Gamma^j\cdot\Gamma^i,$$
$$(\flat_{i}\Gamma^{j})(\flat_{i}\Gamma^{i})^{2}=q^{-c_{ii}-2c_{ij}}\flat_{i}^{2}\flat_{j} \Gamma^i\cdot \Gamma^i\cdot\Gamma^j.$$
In general, we have the following identity:
$$(\flat_{i}\Gamma^{i})^{r}(\flat_{j}\Gamma^{j})(\flat_{i}\Gamma^{i})^{s}=q^{-(r+s)c_{ji}-\frac{(r+s)(r+s-1)}{2}c_{ii}}
\flat_{i}^{r+s}\flat_{j}(\Gamma^i)^{s}\cdot \Gamma^j\cdot(\Gamma^i)^{r}$$
for any two natural numbers $r,s$ satisfying $r+s=1-a_{ij}$.
Therefore, by Proposition \ref{p3.4}, we have:
\begin{eqnarray*}
&&\sum_{r+s=1-a_{ij}}(-1)^{s}\left [
\begin{array}{c} 1-a_{ij}\\s
\end{array}\right]_{d_{i}}((\flat_i\Gamma^i)^{r} (\flat_j\Gamma^{j}) (\flat_i\Gamma^{i})^{s}\\
&&=q^{-(r+s)c_{ji}-\frac{(r+s)(r+s-1)}{2}c_{ii}}
\flat_{i}^{r+s}\flat_{j}\sum_{r+s=1-a_{ij}}(-1)^{s}\left [
\begin{array}{c} 1-a_{ij}\\s
\end{array}\right]_{d_{i}}(\Gamma^i)^{s}\cdot \Gamma^j\cdot(\Gamma^i)^{r}\\
&&=0.
\end{eqnarray*}

\end{proof}

\begin{proposition}\label{p4.6}  In $D(A_{q}(\mathfrak{g}))$, we have the following relations:
\begin{eqnarray}&&(1\bi \chi_{i})h_{j}=h_j(1\bi \chi_{i}),\;\;\;\;(\flat_{i}\chi_{i})^{n}=H_{i}^{-2},\\
&&h_{i}(\flat_j\Gamma^{j})h_{i}^{-1}=\q^{-\delta_{ij}}\flat_j\Gamma^{j},\\
&&(\flat_{i}\chi_{i})e_{j}(\flat_{i}\chi_{i})^{-1}=\q^{c_{ji}}q^{-2c_{ji}}e_{j},\\
&&(\flat_{i}\chi_{i})(\flat_{j}\Gamma^{j})(\flat_{i}\chi_{i})^{-1}=\q^{-c_{ji}}q^{2c_{ji}}(\flat_{j}\Gamma^{j}),\\
&& (\flat_{j}\Gamma^{j})e_{i}-q^{-c_{ji}}e_{i}(\flat_{j}\Gamma^{j})=\delta_{ij}(1\bi \e-H_{i}^{-1}\flat_{i}\chi_i).
\end{eqnarray}
\end{proposition}
\begin{proof} By Formula \eqref{eq;4}, we have $(1\bi \chi_{i})h_{j}=h_{j}\bi (h_{j}^{-1}\rightharpoonup \chi_i\leftharpoonup h_j)=
h_j\bi \chi_{i}=h_{j}(1\bi \chi_{i})$. To show the second equation in (4.7), we use the formula $(\star)$ in Theorem \ref{t1}. We obtain:
\begin{eqnarray*}
&&(1\bi \chi_i)(1\bi \chi_i)\\
&&=\omega^{(3)}\bi(\omega^{(5)}\rightharpoonup \chi_i\leftharpoonup \omega^{(1)})(\omega^{(4)}\rightharpoonup \chi_i\leftharpoonup \omega^{(2)})\\
&&=\sum_{a=b=\epsilon_{i},c\in (\mathbbm{Z}_{n})^{m},d=e=(n-1)\epsilon_{i}}\prod_{k,  j\neq i}\q^{c_{kj}0}\prod_{k\neq i,j=i}\q^{c_{ki}c_{k}[\frac{(n-1)'+(n-1)'}{n}]}\\
&&\;\;\;\;\;\;\;\;\;\;\;\;\;\;\;\;\q^{c_{ii}(-1+c_{i}[\frac{(n-1)'+(n-1)'}{n}]-(n-1))}\mathbf{1}_{c}\bi \chi_{i}^{2}\\
&&=\sum_{c\in (\mathbbm{Z}_{n})^{m}}\prod_{k=1}^{m}\q^{c_{ki}c_{k}}\mathbf{1}_{c}\bi \chi_{i}^{2}\\
&&=H_{i}\chi_{i}^{2}.
\end{eqnarray*}

Here $\chi_i^2$ is the product $\chi_i\cdot \chi_i$ in $M_q(\mathfrak{g})$.
Inductively, we have $(1\bi \chi_i)^{k}=H_{i}^{k-1}\chi_{i}^{k}$ for $1\leq k\leq n$. In particular,
$(1\bi \chi_i)^{n}=H_{i}^{-1}\chi_{i}^{n}=H_{i}^{-1}$. By the first part of (4.7), we obtain: $$(\flat_i\chi_i)^{n}=\flat_i^{n}\chi_i^n=H_{i}^{-1}H_{i}^{-1}=H_{i}^{-2}.$$

For (4.8), it is enough to show that $h_i\Gamma^{j}h_{i}^{-1}=\q^{-\delta_{ij}}(1\bi\Gamma^{j})$ because the elements $h_i$ and $\flat_j$ are commutative.
Indeed, using Formula \eqref{eq;4} we obtain: $$h_i\Gamma^{j}h_{i}^{-1}=h_{i}(h_{i}^{-1}\chi_{j}(h_{i}^{-1})\Gamma^{j})=\q^{-\delta_{ij}}(1\bi\Gamma^{j}).$$

For (4.9), we have the following:
\begin{eqnarray*}
&&(\flat_{i}\chi_{i})e_{j}(\flat_{i}\chi_{i})^{-1}\\
&&=\flat_{i} (e_{j})_{(1)(2)}\omega^{(3)}\bi(\omega^{(5)}\rightharpoonup \e \leftharpoonup \omega^{(1)})(\omega^{(4)}
S((e_j)_{(2)})\rightharpoonup \chi_i\leftharpoonup (e_j)_{(1)(1)}\omega^{(2)})(\flat_{i}\chi_{i})^{-1}\\
&&=\flat_{i}(\sum_{k=1}^{n-1}\mathbf{1}^{j}_{k}e_{j}\bi q^{c_{ji}(n-1)}\chi_i+\mathbf{1}^{j}_{0}e_{j}\bi \q^{-c_{ij}}q^{c_{ji}(n-1)}\chi_i)(\flat_{i}\chi_{i})^{-1}\\
&&=[q^{c_{ji}(n-1)}\sum_{a\in (\mathbbm{Z}_{n})^{m},a_{j}\neq 0}\prod_{l=1}^{m}q^{-c_{il}a_{l}}\mathbf{1}_{a}e_{j}\bi \chi_{i}+\\
&&\;\;\;\;\;\;\;\;\;\;\;\;\;\;\;\;\;\;\;\;\;q^{-c_{ji}}\sum_{a\in (\mathbbm{Z}_{n})^{m},a_{j}= 0}\prod_{l=1}^{m}q^{-c_{il}a_{l}}\mathbf{1}_{a}e_{j}\bi \chi_{i}](\flat_{i}\chi_{i})^{-1}.
\end{eqnarray*}
Now we need the following identities:
\begin{eqnarray*}
&&q^{c_{ji}(n-1)}\sum_{a\in (\mathbbm{Z}_{n})^{m},a_{j}\neq 0}\prod_{l=1}^{m}q^{-c_{il}a_{l}}\mathbf{1}_{a}e_{j}\\
&&=q^{c_{ji}(n-1)}\sum_{a\in (\mathbbm{Z}_{n})^{m},a_{j}\neq 0}\prod_{l=1}^{m}q^{-c_{il}a_{l}}e_{j}\mathbf{1}_{a-\epsilon_j}\\
&&=q^{c_{ji}(n-1)}q^{-c_{ij}}e_{j}\sum_{a\in (\mathbbm{Z}_{n})^{m},a_{j}\neq n-1}\prod_{l=1}^{m}q^{-c_{il}a_{l}}\mathbf{1}_{a},
\end{eqnarray*}
and:
\begin{eqnarray*}
&&q^{-c_{ji}}\sum_{a\in (\mathbbm{Z}_{n})^{m},a_{j}= 0}\prod_{l=1}^{m}q^{-c_{il}a_{l}}\mathbf{1}_{a}e_{j}\\
&&=q^{-c_{ji}}\sum_{a\in (\mathbbm{Z}_{n})^{m},a_{j}= 0}\prod_{l=1}^{m}q^{-c_{il}a_{l}}e_{j}\mathbf{1}_{a-\epsilon_j}\\
&&=q^{-c_{ji}}q^{c_{ij}(n-1)}e_{j}\sum_{a\in (\mathbbm{Z}_{n})^{m},a_{j}= n-1}\prod_{l=1}^{m}q^{-c_{il}a_{l}}\mathbf{1}_{a}.
\end{eqnarray*}
By applying the above identities to the expression of $(\flat_{i}\chi_{i})e_{j}(\flat_{i}\chi_{i})^{-1}$, we obtain:
\begin{eqnarray*}
&&(\flat_{i}\chi_{i})e_{j}(\flat_{i}\chi_{i})^{-1}\\
&&=q^{-c_{ji}}q^{c_{ij}(n-1)}e_{j}(\flat_i\bi \chi_i)(\flat_{i}\chi_{i})^{-1}\\
&&=\q^{c_{ji}}q^{-2c_{ji}}e_{j}.
\end{eqnarray*}
To show (4.10), we only need to verify that $(\flat_{i}\chi_{i})(1\bi\Gamma^{j})(\flat_{i}\chi_{i})^{-1}=\q^{-c_{ji}}q^{2c_{ji}}(1\bi\Gamma^{j})$
since the elements $\chi_i$ and $\flat_j$ are commutative. Note that $(\flat_{i}\chi_{i})^{-1}=H_{i}^{-1}\flat_{i}^{-1}\chi_{i}^{n-1}$. Thus,  we have:
\begin{eqnarray*}
&&(\flat_{i}\chi_{i})(1\bi\Gamma^{j})(\flat_{i}\chi_{i})^{-1}\\
&&=\flat_{i}[\omega^{(3)}\bi(\omega^{(5)}\rightharpoonup \Gamma^{j} \leftharpoonup \omega^{(1)})(\omega^{(4)}
\rightharpoonup \chi_i\leftharpoonup \omega^{(2)})](\flat_{i}\chi_{i})^{-1}\\
&&=(\flat_{i}\q^{-c_{ji}}\bi (\Gamma^{j}\cdot \chi_i))H_{i}^{-1}\flat_{i}^{-1}\chi_{i}^{n-1}\\
&&=\q^{-c_{ji}}\flat_{i}[(1\bi \Gamma^{j}\cdot \chi_i)H_{i}^{-1}\flat_{i}^{-1}\chi_{i}^{n-1}],
\end{eqnarray*}
and: \begin{eqnarray*}
&&(1\bi \Gamma^{j}\cdot \chi_i)(H_{i}^{-1}\flat_{i}^{-1}\bi \chi_{i}^{n-1})\\
&&=\sum_{a,b,c,d,e\in (\mathbbm{Z}_{n})^{m}}\omega_{a,b,c,d,e}\prod_{l=1}^{m}q^{c_{il}(b_l+c_l+d_l)'}H_{i}^{-1}\mathbf{1}_{c}\bi
(S(\mathbf{1}_{e})\rightharpoonup \chi_i^{n-1}\leftharpoonup \mathbf{1}_{a})\\
&&\;\;\;\;\;\;\;\;\;\;\;\;\;\;\;\;\;\;\;(S(H_{i}^{-1}\mathbf{1}_{d})\rightharpoonup \Gamma^{j}\cdot \chi_i\leftharpoonup H_{i}^{-1}\mathbf{1}_{b})\\
&&=\sum_{c\in (\mathbbm{Z}_{n})^{m},a=d=(n-1)\epsilon_{i},b=\epsilon_{i}+\epsilon_{j},e=\epsilon_{i}}\prod_{s\neq i\neq k}\q^{c_{sk}0}
\prod_{s=i\neq k}\q^{c_{ik}(-(n-1)[\frac{b_{k}+c_{k}}{n}])}\\
&&\prod_{s\neq i=k}\q^{c_{si}(b_{s}+c_{s})}\q^{c_{ii}(1+c_{i}-1)}\prod_{l\neq j}q^{c_{il}c_{l}}q^{c_{ij}(1+c_{j})'}\q^{-c_{ji}}H_{i}^{-1}\mathbf{1}_{c}\bi \chi_{i}^{-1}\cdot(\Gamma^{j}\cdot \chi_i)\\
&&=\sum_{c\in (\mathbbm{Z}_{n})^{m}}\q^{c_{ij}[\frac{1+c_{j}}{n}]}\prod_{l\neq j}q^{c_{il}c_{l}}q^{c_{ij}(1+c_{j})'}
\prod_{s=1}^{m}\q^{a_{si}c_{s}}H_{i}^{-1}\mathbf{1}_{c}\bi \chi_{i}^{-1}\cdot(\Gamma^{j}\cdot \chi_i)\\
&&=q^{c_{ij}}\flat^{-1}_{i}\bi \chi_{i}^{-1}\cdot(\Gamma^{j}\cdot \chi_i)\\
&&=q^{c_{ij}+c_{ji}} \flat^{-1}_{i}\bi \Gamma^j,
\end{eqnarray*}
where $\omega_{a,b,c,d,e}$ denotes the coefficient of $\mathbf{1}_{a}\otimes\mathbf{1}_{b}\otimes \mathbf{1}_{c}\otimes S(\mathbf{1}_{d})\otimes S(\mathbf{1}_{e})$ in $\omega$. For the third equality we used $\prod_{s=1}^{m}\q^{a_{si}c_{s}}H_{i}^{-1}\mathbf{1}_{c}=\mathbf{1}_{c}$ and Lemma \ref{l4.1}.  Therefore, $(\flat_{i}\chi_{i})(1\bi\Gamma^{j})(\flat_{i}\chi_{i})^{-1}=\q^{-c_{ji}}q^{2c_{ji}}(1\bi \Gamma^{j})$.

We arrive now at the proof of the last equality, (4.11). Using the comultiplication formula for $e_{i}$
given in Lemma \ref{l2}, we have:
\begin{eqnarray*}
(\Delta\otimes id)\Delta(e_{i})&=&e_{i}\otimes \flat^{-1}_i\otimes \flat^{-1}_i+1\otimes \sum_{k=1}^{n-1}\mathbf{1}^{i}_{k}e_{i}\otimes \flat^{-1}_i
+H_{i}^{-1}\otimes \mathbf{1}^{i}_{0}e_{i} \otimes \flat^{-1}_i\\
&&+1\otimes 1\otimes \sum_{k=1}^{n-1}\mathbf{1}^{i}_{k}e_{i}
+H_{i}^{-1}\otimes H_{i}^{-1}\otimes \mathbf{1}^{i}_{0}e_{i}.
\end{eqnarray*}
Substituting the above comultiplication of $e_i$ in the following equation:
$$(1\bi \Gamma^{j})(e_{i}\bi \e)=(e_{i})_{(1)(2)}\bi S((e_{i})_{(2)})\rightharpoonup \Gamma^{j}\leftharpoonup (e_{i})_{(1)(1)},$$
we obtain:
$$(1\bi \Gamma^{j})(e_{i}\bi \e)=\sum_{k=1}^{n-1}\mathbf{1}^{i}_{k}e_{i}\bi \Gamma^{j}+\q^{-c_{ji}}\mathbf{1}^{i}_{0}e_{i}\bi\Gamma^{j},\  \mathrm{for}\ i\not=j.$$
Now multiplying both sides of the above identity with the element $\flat_j$, we obtain:
\begin{eqnarray*}
(\flat_j \Gamma^{j})e_{i}&=&e_{i}\sum_{a\in (Z)_{n}^{m},a_{i}\neq 0}\prod_{l=1}^{m}q^{-c_{jl}a_{l}}\mathbf{1}_{a-\epsilon_{i}}\Gamma^{j}
+\q^{-c_{ji}}e_{i}\sum_{a\in (Z)_{n}^{m},a_{i}= 0}\prod_{l=1}^{m}q^{-c_{jl}a_{l}}\mathbf{1}_{a-\epsilon_{i}}\Gamma^{j}\\
&=&q^{-c_{ji}}e_{i}\sum_{a\in (Z)_{n}^{m},a_{i}\neq n-1}\prod_{l=1}^{m}q^{-c_{jl}a_{l}}\mathbf{1}_{a}\Gamma^{j}+
q^{-c_{ji}}e_{i}\sum_{a\in (Z)_{n}^{m},a_{i}= n-1}\prod_{l=1}^{m}q^{-c_{jl}a_{l}}\mathbf{1}_{a}\Gamma^{j}\\
&=& q^{-c_{ji}}e_{i}(\flat_{j}\Gamma^j).
\end{eqnarray*}
If $i=j$, then we have:
$$(1\bi \Gamma^{i})(e_{i}\bi \e)=\flat_{i}^{-1}+
\sum_{k=1}^{n-1}\mathbf{1}^{i}_{k}e_{i}\bi\Gamma^{i}+\q^{-c_{ii}}\mathbf{1}^{i}_{0}e_{i}\bi\Gamma^{i}-H_{i}^{-1}\bi\chi_i.$$
Similarly, by multiplying both sides with the element $\flat_i$, we obtain:
$$(\flat_i \Gamma^{i})e_{i}=1\bi\e+q^{-c_{ii}}e_{i}(\flat_{i}\Gamma^i)-H_{i}^{-1}(\flat_i\chi_i).$$
\end{proof}

The next step is to determine the coalgebraic structure of $D(A_q(\mathfrak{g}))$. We divide it in the following two propositions.
\begin{proposition}\label{p4.7} In $D(A_{q}(\mathfrak{g}))$, we have:
     \begin{eqnarray} && \D(h_{i})=h_{i}\otimes h_{i},\;\;\Delta(e_{i})=e_{i}\otimes \flat_{i}^{-1}+1\otimes \sum_{j=1}^{n-1}\mathbf{1}^{i}_{j}e_{i}+H_{i}^{-1}\otimes \mathbf{1}^{i}_{0}e_{i},\\
      &&\e(h_{i})=1,\;\;\;\;\;\;\;\e(e_{i})=0,
      \end{eqnarray}
      for $1\leq i\leq m$.
\end{proposition}
\begin{proof} Due to the fact that
$A_q(\mathfrak{g})$ is a quasi-Hopf subalgebra of $D(A_{q}(\mathfrak{g}))$.
\end{proof}

\begin{proposition}\label{p4.8} In $D(A_{q}(\mathfrak{g}))$, we have:
\begin{eqnarray} && \D(\flat_i\chi_i)=\flat_i\chi_i\otimes \flat_i\chi_i,\;\;\\
&&\Delta(\flat_i\Gamma^{i})=\flat_i\Gamma^{i}\otimes \flat_{i}+H_{i}^{-1}(\flat_i\chi_i)\otimes (\flat_i\Gamma^{i})\sum_{j=1}^{n-1}\mathbf{1}^{i}_{j}+(\flat_i\chi_i)\otimes (\flat_i\Gamma^{i})\mathbf{1}^{i}_{0},\\
      &&\e(\flat_i\chi_i)=1,\;\;\;\;\;\;\;\e(\flat_i\Gamma^{i})=0,
      \end{eqnarray}
      for $1\leq i\leq m$.
\end{proposition}
\begin{proof} Recall that, for any $\psi\in M_{q}(\mathfrak{g})$, we defined $\textbf{T}(\psi)=\phi^{(1)}_{(2)}\bi S(\phi^{(2)})\alpha \phi^{(3)}\rightharpoonup \psi\leftharpoonup \phi^{(1)}_{(1)}$ (see Equation \eqref{eq;1.13} before Theorem \ref{t1}). Thus:
\begin{eqnarray*}
\textbf{T}(\chi_i)&=& \phi^{(1)}_{(2)}\bi S(\phi^{(2)})\alpha \phi^{(3)}\rightharpoonup \chi_i\leftharpoonup \phi^{(1)}_{(1)}\\
&=& \sum_{a^1,a^2,b,c\in(\mathbbm{Z}_{n})^{m}}\prod_{s,t=1}^{m}\q^{-c_{st}(a^1_s+a^2_s)[\frac{b_t+c_t}{n}]}\mathbf{1}_{a^{2}}\bi
S(\mathbf{1}_{b})\alpha\mathbf{1}_{c}\rightharpoonup \chi_i\leftharpoonup \mathbf{1}_{a^{1}}\\
&=&  \sum_{a^1,a^2,b,c\in(\mathbbm{Z}_{n})^{m}}\prod_{s,t=1}^{m}\q^{-c_{st}(a^1_s+a^2_s)[\frac{b_t+c_t}{n}]}\q^{c_{st}c_{s}[\frac{n-1+c_{t}}{n}]}
\mathbf{1}_{a^{2}}\bi
S(\mathbf{1}_{b})\mathbf{1}_{c}\rightharpoonup \chi_i\leftharpoonup \mathbf{1}_{a^{1}}\\
&=&\sum_{a^1=c=\epsilon_{i},b=(n-1)\epsilon_{i},a^2\in(\mathbbm{Z}_{n})^{m}}
\prod_{s=1}^{m}\q^{-c_{si}a^2_s}\q^{-c_{ii}}\q^{c_{ii}}
\mathbf{1}_{a^{2}}\bi\chi_i\\
&=&\sum_{a\in(\mathbbm{Z}_{n})^{m}}
\prod_{s=1}^{m}\q^{-c_{si}a_s}\mathbf{1}_{a}\bi\chi_i\\
&=&H_{i}^{-1}\chi_i.
\end{eqnarray*}
Applying formula $(\star\star)$ in Theorem \ref{t1}, we have:
\begin{eqnarray*}
&&\Delta(\mathbf{T}(\chi_i))\\
&&= \tilde{\phi}^{(2)}\mathbf{T}(\chi_i\leftharpoonup \tilde{\phi}^{(1)})\phi^{(-1)}\phi^{(1)}
 \otimes\tilde{\phi}^{(3)}
\phi^{(-3)}\mathbf{T}(\phi^{(3)}\rightharpoonup \chi_i\leftharpoonup \phi^{(-2)})\phi^{(2)}\\
&&=\sum_{a^1,a^2,a^3,b^1,b^2,b^3,c^1,c^2,c^3\in(\mathbbm{Z}_{n})^{m}}\prod_{s,t=1}^{m}\q^{-c_{st}a^1_s[\frac{a^{2}_t+a^{3}_t}{n}]
+c_{st}b^1_s[\frac{b^{2}_t+b^{3}_t}{n}]-c_{st}c^1_s[\frac{c^{2}_t+c^{3}_t}{n}]}\\
&& \;\;\;\;\;\;\;\;\;\;\;\;\mathbf{1}_{a^{2}}\mathbf{T}(\chi_i\leftharpoonup \mathbf{1}_{a^{1}})\mathbf{1}_{b^{1}}\mathbf{1}_{c^{1}}\otimes
\mathbf{1}_{a^{3}}\mathbf{1}_{b^{3}}\mathbf{T}(\mathbf{1}_{c^{3}}\rightharpoonup \chi_i\leftharpoonup \mathbf{1}_{b^{2}})\mathbf{1}_{c^{2}}\\
&&=\sum_{a^1,a^2,a^3,b^2,c^3\in(\mathbbm{Z}_{n})^{m}}\prod_{s,t=1}^{m}\q^{-c_{st}a^1_s[\frac{a^{2}_t+a^{3}_t}{n}]
+c_{st}a^2_s[\frac{b^{2}_t+a^{3}_t}{n}]-c_{st}a^2_s[\frac{a^{3}_t+c^{3}_t}{n}]}\\
&& \;\;\;\;\;\;\;\;\;\;\;\;\mathbf{T}(\chi_i\leftharpoonup \mathbf{1}_{a^{1}})\mathbf{1}_{a^{2}}\otimes
\mathbf{T}(\mathbf{1}_{c^{3}}\rightharpoonup \chi_i\leftharpoonup \mathbf{1}_{b^{2}})\mathbf{1}_{a^{3}}\\
&&=\sum_{a^2,a^3\in(\mathbbm{Z}_{n})^{m}}\prod_{t=1}^{m}\q^{-c_{it}[\frac{a^{2}_t+a^{3}_t}{n}]}
\prod_{s=1}^{m}\q^{c_{si}a^2_s[\frac{1+a^{3}_i}{n}]}\prod_{s=1}^{m}\q^{-c_{si}a^2_s[\frac{a^{3}_i+1}{n}]}\\
&&\;\;\;\;\;\;\;\;\;\;\;\;\;
\mathbf{T}(\chi_i)\mathbf{1}_{a^{2}}\otimes \mathbf{T}(\chi_i)\mathbf{1}_{a^{3}}\\
&&=\sum_{a^2,a^3\in(\mathbbm{Z}_{n})^{m}}\prod_{t=1}^{m}\q^{-c_{it}[\frac{a^{2}_t+a^{3}_t}{n}]}\mathbf{T}(\chi_i)\mathbf{1}_{a^{2}}\otimes \mathbf{T}(\chi_i)\mathbf{1}_{a^{3}}.
\end{eqnarray*}
Therefore, we obtain:
\begin{eqnarray*}
\Delta(\flat_i\mathbf{T}(\chi_i))&=&\Delta(\flat_i)\Delta(\mathbf{T}(\chi_i))\\
&=&(\sum_{b,c\in(\mathbbm{Z}_{n})^{m}}\prod_{t=1}^{m}q^{-c_{it}(b_{t}+c_{t})'}\mathbf{1}_{b}\otimes \mathbf{1}_{c})\\
&&(\sum_{b,c\in(\mathbbm{Z}_{n})^{m}}\prod_{t=1}^{m}\q^{-c_{it}[\frac{b_t+c_t}{n}]}\mathbf{T}(\chi_i)\mathbf{1}_{b}\otimes \mathbf{T}(\chi_i)\mathbf{1}_{c})\\
&=&\sum_{b,c\in(\mathbbm{Z}_{n})^{m}}\prod_{t=1}^{m}q^{-c_{it}(b_{t}+c_{t})}\mathbf{1}_{b}\mathbf{T}(\chi_i)\otimes \mathbf{1}_{c}\mathbf{T}(\chi_i)\\
&=&\flat_i\mathbf{T}(\chi_i)\otimes \flat_i\mathbf{T}(\chi_i).
\end{eqnarray*}
This means that $\flat_i\mathbf{T}(\chi_i)$ is a group-like element. Since $\mathbf{T}(\chi_i)=H_{i}^{-1}\chi_i$ and
$H_{i}^{-1}$ is a group-like element, $\flat_i\chi_i$ is group-like too. Thus the proof of (4.14) is done.

Using the same method, one can show that:
$$\mathbf{T}(\Gamma^i)=1\bi \Gamma^i$$
and:
\begin{eqnarray*}
\Delta(\mathbf{T}(\Gamma^i))&=&\sum_{a^2,a^3\in(\mathbbm{Z}_{n})^{m}}\prod_{t=1}^{m}\q^{-c_{it}[\frac{a^{2}_t+a^{3}_t}{n}]}
\mathbf{1}_{a^{2}}\mathbf{T}(\Gamma^i)\otimes \mathbf{1}_{a^{3}}\\
&&+ \sum_{a^2,a^3\in(\mathbbm{Z}_{n})^{m}}\prod_{t=1}^{m}\q^{-c_{it}[\frac{a^{2}_t+a^{3}_t}{n}]}\prod_{s=1}^{m}\q^{c_{si}a^{2}_{s}
[\frac{1+a^{3}_{i}}{n}]}
\mathbf{1}_{a^{2}}\mathbf{T}(\chi_i)\otimes \mathbf{1}_{a^{3}}\mathbf{T}(\Gamma^i).
\end{eqnarray*}
Thus:
\begin{eqnarray*}
\Delta(\flat_i\mathbf{T}(\Gamma^i))&=&\Delta(\flat_i)\Delta(\mathbf{T}(\Gamma^i))\\
&=&\sum_{a^2,a^3\in(\mathbbm{Z}_{n})^{m}}\prod_{t=1}^{m}q^{-c_{it}(a^{2}_{t}+a^{3}_{t})}
\mathbf{1}_{a^{2}}\mathbf{T}(\Gamma^i)\otimes \mathbf{1}_{a^{3}}\\
&&+H_{i}^{-1}\sum_{a^2\in(\mathbbm{Z}_{n})^{m}}\prod_{t=1}^{m}\q^{-c_{it}a^{2}_{t}}\mathbf{1}_{a^{2}}\chi_i
\otimes \sum_{a^3; a^{3}_{i}\neq n-1} \prod_{t=1}^{m}\q^{-c_{it}a^{3}_{t}}\mathbf{1}_{a^{3}}\mathbf{T}(\Gamma^i)\\
&&+\sum_{a^2\in(\mathbbm{Z}_{n})^{m}}\prod_{t=1}^{m}\q^{-c_{it}a^{2}_{t}}\mathbf{1}_{a^{2}}\chi_i
\otimes \sum_{a^3; a^{3}_{i}= n-1} \prod_{t=1}^{m}\q^{-c_{it}a^{3}_{t}}\mathbf{1}_{a^{3}}\mathbf{T}(\Gamma^i)\\
&=&\flat_i\Gamma^{i}\otimes \flat_{i}+H_{i}^{-1}(\flat_i\chi_i)\otimes (\flat_i\Gamma^{i})\sum_{j=1}^{n-1}\mathbf{1}^{i}_{j}+(\flat_i\chi_i)\otimes (\flat_i\Gamma^{i})\mathbf{1}^{i}_{0}.
\end{eqnarray*}

Now (4.16) is clear.
\end{proof}

Finally, we determine the reassociator $\phi$, the elements $\alpha,\beta$ and the antipode $S$ for $D(A_{q}(\mathfrak{g}))$.
\begin{proposition} In $D(A_{q}(\mathfrak{g}))$, the reassociator is given by:
\begin{equation}\phi=\sum_{a,b,c\in (\mathbbm{Z}_{n})^{m}}(\prod_{i,j=1}^{m}\q^{-c_{ij}a_{i}[\frac{b_{j}+c_{j}}{n}]})\mathbf{1}_{a}
 \otimes \mathbf{1}_b\otimes \mathbf{1}_{c}. \end{equation}
 The elements $\alpha,\beta$ can be chosen as:
 \begin{equation} \alpha=\sum_{a\in (\mathbbm{Z}_{n})^{m}}\prod_{s,t=1}^{m}\q^{c_{st}a_{s}[\frac{n-1+a_{t}}{n}]}\mathbf{1}_{a},\;\;\;\;\beta=1.
 \end{equation}
 The antipode $S$ is determined by:
\begin{eqnarray} && S(h_{i})=h_{i}^{-1},\;\;\;\;\;S(\flat_i\chi_i)=(\flat_i\chi_i)^{-1},\\
&&S(e_{i})=-(\alpha\sum_{j=1}^{n-1}\mathbf{1}^{i}_{j}e_{i}+H_{i}\alpha \mathbf{1}^{i}_{0}e_{i})\flat_{i}\alpha^{-1},\\
&&S(\flat_i\Gamma^i)=-(H_{i}(\flat_i\chi_i)^{-1}\alpha (\flat_i\Gamma^i)\sum_{j=1}^{n-1}\mathbf{1}^{i}_{j}+(\flat_i\chi_i)^{-1}\alpha (\flat_i\Gamma^i)\mathbf{1}^{i}_{0})\flat_i^{-1}\alpha^{-1},
      \end{eqnarray}
      for $1\leq i\leq m$.
\end{proposition}
\begin{proof} By Theorem \ref{t1} (2) and Lemma \ref{l2}, the reassociator $\phi$ is given by (4.17), and $\alpha, \beta$ can be chosen
as  in (4.18). Since the elements  $h_i$ and $\flat_i\chi_i$ are group-like, (4.19) is obvious. Both (4.20) and (4.21) follow directly from the definition of the antipode and the comultiplication formulas for
$e_{i}$ and $\flat_{i}\Gamma^{i}$.
\end{proof}

\section{Presentation of quasi-Frobenius-Lusztig kernels}
In this section, we present $D(A_{q}(\mathfrak{g}))$ in terms of generators and relations.
Let $\mathfrak{g}$ be a simple Lie algebra of finite type, $A=(a_{ij})_{m\times m}$ its Cartan matrix
and $C=(d_{i}a_{ij})=(c_{ij})$ the symmetrized Cartan matrix. Let $n$ be a natural number $\geq 4$, and $q$ an $n^{2}$-th primitive root of unity, $\q=q^{n}$ and $l_{i}=\ord(q^{c_{ii}})$.

\begin{definition} The quasi-Frobenius-Lusztig kernel ${\Q}\mathbf{u}_{q}(\mathfrak{g})$ is a  quasi-Hopf algebra defined as follows.
  As an associative algebra, it is generated by $E_{i},F_{i},K_{i},\hat{K}_{i}\;(1\leq i\leq m)$ satisfying:
    \begin{gather}\label{eq;5.1} K_{i}K_{j}=K_{j}K_{i},\;\;\hat{K}_{i}\hat{K}_{j}=\hat{K}_{j}\hat{K}_{i},\;\;
    K_{i}\hat{K}_{j}=\hat{K}_{j}K_{i},\\
   \label{eq;5.2} K_i^{n}=1,\;\;\hat{K}_i^{n}=\prod_{l=1}^{m}K_{l}^{-2c_{il}},\\
   \label{eq;5.3} K_i E_j=\q^{\delta_{ij}}E_j K_i,\;\;K_i F_j=\q^{-\delta_{ij}}F_j K_i,\\
   \label{eq;5.4} \hat{K}_i E_j=\q^{c_{ij}}q^{-2c_{ij}} E_j \hat{K}_i,\;\;\hat{K}_i F_j=\q^{-c_{ij}}q^{2c_{ij}} F_j \hat{K}_i,\\
    \label{eq;5.5}F_{j}E_i-q^{-c_{ij}}E_{i}F_j=\delta_{ij}(1-\prod_{l=1}^{m}K_{l}^{-c_{il}}\hat{K}_i),\\
    E_{i}^{l_{i}}=F_{i}^{l_{i}}=0,
\end{gather}
\begin{gather}
    \left \{
\begin{array}{ll} \sum_{r+s=1-a_{ij}}(-1)^{s}\left [
\begin{array}{c} 1-a_{ij}\\s
\end{array}\right]_{d_{i}}E_{i}^{r}E_{j}E_{i}^{s}=0 & \;\;\;\;i\neq j\\
\sum_{r+s=1-a_{ij}}(-1)^{s}\left [
\begin{array}{c} 1-a_{ij}\\s
\end{array}\right]_{d_{i}}F_{i}^{r}F_{j}F_{i}^{s}=0 &
\;\;\;\;i\neq j.
\end{array}\right.
\end{gather}
 for $1\leq i,j\leq m$.

 \emph{Let $\{\mathbf{1}_{a}|a=(a_{1},\ldots,a_{m})\in (\mathbbm{Z}_{n})^{m}\}$ be the set of primitive idempotents of the group algebra of $\langle K_{i}|1\leq i\leq m\rangle\cong (\mathbbm{Z}_{n})^{m}$, $\mathbf{1}^{i}_{k}:=\frac{1}{n}\sum_{j=0}^{n-1}(\q^{n-k})^{j}K_{i}^{j}$, $\flat_i:=\sum_{a\in (\mathbbm{Z}_{n})^{m}}\prod_{j=1}^{m}q^{-c_{ij}a_{j}}\mathbf{1}_{a},\;\;H_{i}:=\prod_{j=1}^{m}K_{j}^{c_{ji}}$.}

 The reassociator $\phi$, the comultiplication $\D$, the counit $\e$, the elements $\alpha,\beta$ and the antipode $S$
    are given by

    \begin{gather} \phi=\sum_{a,b,c\in (\mathbbm{Z}_{n})^{m}}(\prod_{i,j=1}^{m}\q^{-c_{ij}a_{i}[\frac{b_{j}+c_{j}}{n}]})\mathbf{1}_{a}
 \otimes \mathbf{1}_b\otimes \mathbf{1}_{c},\\
    \D(K_{i})=K_{i}\otimes K_{i},\;\;\;\;\D(\hat{K}_{i})=\hat{K}_{i}\otimes \hat{K}_{i},\\
    \Delta(E_{i})=E_{i}\otimes \flat_{i}^{-1}+1\otimes \sum_{j=1}^{n-1}\mathbf{1}^{i}_{j}E_{i}+H_{i}^{-1}\otimes \mathbf{1}^{i}_{0}E_{i},
 \end{gather}

 \begin{gather}
    \D(F_{i})=F_i\otimes \flat_{i}+H_{i}^{-1}\hat{K}_{i}\otimes F_i\sum_{j=1}^{n-1}\mathbf{1}^{i}_{j}+\hat{K}_i\otimes F_i\mathbf{1}^{i}_{0},\\
    \e(K_{i})=\e(\hat{K}_{i})=1,\;\;\;\;\e(E_i)=\e(F_i)=0,\\
    \alpha=\sum_{a\in (\mathbbm{Z}_{n})^{m}}\prod_{s,t=1}^{m}\q^{c_{st}a_{s}[\frac{n-1+a_{t}}{n}]}\mathbf{1}_{a},\;\;\;\;\beta=1\\
    S(K_{i})=K_{i}^{-1},\;\;\;\;S(\hat{K}_{i})=\hat{K}_{i}^{-1},\\
    S(E_{i})=-(\alpha\sum_{j=1}^{n-1}\mathbf{1}^{i}_{j}E_{i}+H_{i}\alpha \mathbf{1}^{i}_{0}E_{i})\flat_{i}\alpha^{-1},\\
    S(F_i)=-(H_{i}\hat{K}_i^{-1}\alpha F_{i}\sum_{j=1}^{n-1}\mathbf{1}^{i}_{j}+\hat{K}_{i}^{-1}\alpha F_{i}\mathbf{1}^{i}_{0})\flat_i^{-1}\alpha^{-1}.
    \end{gather}
    for $1\leq i\leq m$.
\end{definition}

\begin{lemma} \label{l5.2} ${\Q}\mathbf{u}_{q}(\mathfrak{g})$ is finite dimensional and $\dim({\Q}\mathbf{u}_{q}(\mathfrak{g}))=(\dim(A_{q}(\mathfrak{g})))^{2}$.
\end{lemma}
\begin{proof} We give a rough proof of this statement.
At first, by the relations (5.3)-(5.5), ${\Q}\mathbf{u}_{q}(\mathfrak{g})$ has an triangle decomposition:
$${\Q}\mathbf{u}_{q}(\mathfrak{g})=\mathbf{u}^+\mathbf{u}^{0}\mathbf{u}^{-}$$
where $\mathbf{u}^+$ (resp. $\mathbf{u}^-$) is the subalgebra generated by $E_{i}$ (resp. $F_{i}$) for $ 1\leq i\leq m$, and
$\mathbf{u}^{0}$ is the subalgebra generated by $K_{i},\hat{K}_{i}$ for $ 1\leq i\leq m$. It is not hard
to see that $\dim(\mathbf{u}^0)=n^{2m}$ and $\dim(\mathbf{u}^+)n^{m}=\dim(\mathbf{u}^-)n^{m}=
\dim(A_{q}(\mathfrak{g}))$. Therefore,
$$\dim({\Q}\mathbf{u}_{q}(\mathfrak{g}))=(\dim(A_{q}(\mathfrak{g})))^{2}.$$
\end{proof}

The following theorem shows  that ${\Q}\mathbf{u}_{q}(\mathfrak{g})$ is a quasi-Hopf algebra though one can also verify that (5.1)-(5.16) define a quasi-Hopf algebra.

\begin{theorem} \label{t5.3} As quasi-Hopf algebras, $D(A_{q}(\mathfrak{g}))\cong {\Q}\mathbf{u}_{q}(\mathfrak{g})$.
\end{theorem}
\begin{proof} Define a map
\begin{eqnarray}\Upsilon:\;{\Q}\mathbf{u}_{q}(\mathfrak{g})\To D(A_{q}(\mathfrak{g})),
&&K_{i}\mapsto h_{i},\;\; \hat{K}_{i}\mapsto \flat_i\chi_i,\notag\\
&&E_i\mapsto e_i,\;\;F_i\mapsto \flat_i\Gamma^i.\notag
\end{eqnarray}
By Propositions \ref{p4.3}, \ref{p4.5} and \ref{p4.6}, $\Upsilon$ is an algebra morphism. By Propositions \ref{p4.7} and
 \ref{p4.8}, $\Upsilon$  preserves the comultiplication. Thanks to  Theorem \ref{t1} (1), $\Upsilon$ is surjective, and hence bijective as the dimensions of the two algebras are equal (see Lemma \ref{l5.2}).
\end{proof}

\section{Twist equivalence}

In this section, we determine when the  quasi-Hopf algebra ${\Q}\mathbf{u}_{q}(\mathfrak{g})$ is not twisted equivalent to a Hopf algebra.

\begin{definition} \emph{(1)} We call a quasi-Hopf algebra $H$  \emph{twist equivalent} to another quasi-Hopf algebra $K$ if there is a twist $J$ of $H$ such that $K\cong H_{J}$ as quasi-bialgebras.

 \emph{(2)}  A quasi-Hopf algebra $H$ is said to be \emph{genuine}  if $H$ is not twist equivalent to any ordinary Hopf algebra.
 \end{definition}

We give various sufficient conditions for ${\Q}\mathbf{u}_{q}(\mathfrak{g})$  to be genuine.

\begin{theorem}\label{t6.1} Assume $\mathfrak{g}$ is of type $A_{m}$ for $m\geq 2$.

\emph{(1)} If $(m+1)|n$, then ${\Q}\mathbf{u}_{q}(\mathfrak{g})$ is a genuine quasi-Hopf algebra.

\emph{(2)} If $m$ is odd and $4|n$, then ${\Q}\mathbf{u}_{q}(\mathfrak{g})$ is a genuine quasi-Hopf algebra.
\end{theorem}
\begin{proof}
 (1) Let $d= \frac{n}{m+1}$ and $\zeta_{m+1}=\q^{d}$. Let $G:=\langle K_{i}|1\leq i\leq m\rangle$ be the subgroup generated by $K_i$'s in ${\Q}\mathbf{u}_{q}(\mathfrak{g})$. Consider the following $1$-dimensional representation of $G$:
 $$\rho:\; G\To \k,\;\;K_{i}\mapsto \zeta_{m+1}^{i}.$$
  We show that $\rho$ can be extended to a $1$-dimensional representation of ${\Q}\mathbf{u}_{q}(\mathfrak{g})$, still denoted by $\rho$. Indeed, we may define:
 $$\rho:\; {\Q}\mathbf{u}_{q}(\mathfrak{g})\To \k,\;\;K_{i}\mapsto \zeta_{m+1}^{i},\;\;\hat{K}_{i}\mapsto 1,\;\;
 E_{i}\mapsto 0,\;\; F_{i}\mapsto 0,$$
 for $1\leq i\leq m$. We need to show that $\rho$ is a well-defined algebra morphism. By our choice, we have $\rho(H_{i})=\rho(\prod_{j=1}^{m}(K_j)^{c_{ji}})=1$. Therefore, the relations \eqref{eq;5.2} and \eqref{eq;5.5} are preserved by $\rho$. The other relations can be checked easily. Thus $\rho$  is  well-defined.

 Now let $X$ be this $1$-dimensional  ${\Q}\mathbf{u}_{q}(\mathfrak{g})$-module and $\langle X\rangle$ be the tensor subcategory generated by $X$.
   Define:
\begin{gather*}X^{\stackrel{\rightharpoonup}{\otimes l}}=:\stackrel{l}{\overbrace{(\cdots(X\otimes X)\otimes X)\cdots)}}.\end{gather*}
Then the objects of $\langle X\rangle$ are direct sums of elements  in $\{X^{\stackrel{\rightharpoonup}{\otimes l}}|0\leq l< m+1\}$.
   Now assume that ${\Q}\mathbf{u}_{q}(\mathfrak{g})$  is twist equivalent to a Hopf algebra. By the general principle of Tannaka-Krein duality (see, e.g., \cite{CE}), there is a fiber functor
   from the category Rep-${\Q}\mathbf{u}_{q}(\mathfrak{g})$ to the category of $\k$-spaces. Thus its restriction to $\langle X\rangle$ is still a fiber functor. This implies that
   the restriction of $\phi$ to $\langle X\rangle$ should come from a 3-coboundary of $(\mathbbm{Z}_{m+1})^{m}$.
   In fact, by the definition of $\rho$, $\mathbf{1}_{a}X\neq 0$ if and only if $kd|a_{k}$ for $1\leq k\leq m$,  and hence
 $$\phi|_{\langle X\rangle}=\sum_{a,b,c\in (\mathbbm{Z}_{m+1})^{m}} \prod_{s,t=1}^{m}\zeta_{m+1}^{-c_{st}sa_{s}}[\frac{tb_{t}+t c_{t}}{m+1}]\mathbf{1}_{a}\otimes \mathbf{1}_{b}\otimes \mathbf{1}_{c}.$$
 Here $\mathbf{1}_{x},\; x=a,b,c$, denotes a primitive element in $\k((\mathbbm{Z}_{m+1})^{m})$. This corresponds to a $3$-cocycle $\Phi$ over
 $((\mathbbm{Z}_{m+1})^{m})^{\wedge}$, the character group of $(\mathbbm{Z}_{m+1})^{m}$. By definition, $\Phi(\chi_{a},\chi_{b},\chi_{c})=\prod_{s,t=1}^{m}\zeta_{m+1}^{-c_{st}sa_{s}[\frac{tb_{t}+t c_{t}}{m+1}]}$
 where $\chi_{x}, x=a,b,c$, is the dual element of $\mathbf{1}_{x}$.

 Now we show that $\Phi$ is not a coboundary and thus we get a contradiction. By Corollary \ref{c1.7}, it is enough to
 compute $F_{3}^{\ast}(\Phi)$. We use the same notations as in Subsection 1.4. We have the following:
 \begin{eqnarray*}
f_{1,1,1}&=&F_{3}^{\ast}(\Phi)(\Psi_{1,1,1})\\
&=&\prod_{l=0}^{m}\Phi(\chi_{\epsilon_{1}},\chi_{l\epsilon_{1}},\chi_{\epsilon_{1}})\\
&=&\zeta_{m+1}^{-c_{11}}\neq 1.
\end{eqnarray*}
By Corollary \ref{c1.7} and Lemma \ref{l1.5},  $\Phi$ is not a coboundary.

(2) Consider the following $1$-dimensional representation of $G$:

\begin{equation}\label{onerep}
\rho:\;G\To \k,\;\; \left\{ \begin{array}{ll}
K_{i}\mapsto 1, & \textrm{if}\ i\ \textrm{is even}\\
K_i\mapsto \q^{\frac{n}{4}}, & \textrm{if}\ i \ \textrm{is odd}.
\end{array}\right.
\end{equation}

It is not hard to see that $\rho(H_i)=\pm 1$, and that $\rho$ can be also extended to a ${\Q}\mathbf{u}_{q}(\mathfrak{g})$-module by setting:
\begin{equation}\label{oneact}
\rho(E_i)=\rho(F_i)=0,\;\;\rho(\hat{K}_i)=\rho(H_{i})
\end{equation}
for $1\leq i\leq m$. Thus using the same argument developed in the proof of (1), we get the desired result.
  \end{proof}

  \begin{theorem} The following hold.

  \emph{(1)} Assume $\mathfrak{g}$ is of type $B_{m}$. If either $2|m$ and $4|n$ or $2\nmid m$ and $8|n$, then ${\Q}\mathbf{u}_{q}(\mathfrak{g})$ is genuine;

  \emph{(2)} Assume $\mathfrak{g}$ is of type $C_{m}$, or $D_{m}$, or $E_{7}$. If $4|n$, then ${\Q}\mathbf{u}_{q}(\mathfrak{g})$ is genuine;

  \emph{(3)} Assume $\mathfrak{g}$ is of type $E_{6}$. If $3|n$, then ${\Q}\mathbf{u}_{q}(\mathfrak{g})$ is genuine.

 \end{theorem}
  \begin{proof} The proof is almost the same as the proof of Theorem \ref{t6.1}. Therefore, we only provide the construction of the $1$-dimensional modules.

 The principle for the construction of such a $1$-dimensional module is: For a $1$-dimensional ${\Q}\mathbf{u}_{q}(\mathfrak{g})$-module, the actions of $E_{i}$ and $F_{i}$
  must be trivial since they are nilpotent. Thus Relation \eqref{eq;5.5} implies that the action of $H_{i}^{-1}\hat{K}_i$ is trivial as well. Applying Relation \eqref{eq;5.2}, the action of
  $H_{i}^{-2}$ is also trivial and hence the action of $H_{i}$ must be $\pm 1$. So a necessary condition for a $1$-dimensional $\k G$-module to be extended
   to a ${\Q}\mathbf{u}_{q}(\mathfrak{g})$-module is: the action of $H_{i}$ is $\pm 1$.  Conversely, given a $1$-dimensional $\k G$-module $(V,\rho)$ satisfying $\rho(H_{i})=\pm 1$,  it can be extended to a ${\Q}\mathbf{u}_{q}(\mathfrak{g})$-module if we set
  $\rho(E_i)=\rho(F_{i})=0$ and  $\rho(\hat{K}_i)=\rho(H_{i})$.  \\[1.5mm]
 \textbf{ Type $B_{m}$}:  If $m$ is even and $4|n$, the 1-dim representation $\rho$ in (\ref{onerep}) extends to a ${\Q}\mathbf{u}_{q}(\mathfrak{g})$-module as in the proof of Theorem \ref{t6.1}.

If $m$ is odd and $8|n$, define:
$$
\rho:\;G\To \k,\;\; \left\{ \begin{array}{ll}
K_{i}\mapsto 1, & \textrm{if}\ i\ \textrm{is even}\\
K_i\mapsto \q^{\frac{n}{4}}, & \textrm{if}\ i \ \textrm{is odd and}\ i\neq m,\\
K_{m}\mapsto \q^{\frac{n}{8}}.&
\end{array}\right.
$$
 Then $\rho$ extends to a ${\Q}\mathbf{u}_{q}(\mathfrak{g})$-module by adding (\ref{oneact}).

\textbf{Type $C_{m}$}: If $4|n$,   define:
$$\rho:\;G\To \k,\;\; \left\{ \begin{array}{ll}
K_{i}\mapsto 1, & \mathrm{for}\ i< m,\\
K_m\mapsto \q^{\frac{n}{4}}. &
\end{array}\right.$$
Then $\rho$ extends to a ${\Q}\mathbf{u}_{q}(\mathfrak{g})$-module by adding (\ref{oneact}).

\textbf{Type $D_{m}$}: If $4|n$,   define:
$$\rho:\;G\To \k,\;\; \left\{ \begin{array}{ll}
K_{i}\mapsto 1, & \mathrm{if}\ i<  m-1,\\
K_i\mapsto \q^{\frac{n}{4}}, & \mathrm{if}\ i=m-1, m.
\end{array}\right.$$
Then $\rho$ extends to a ${\Q}\mathbf{u}_{q}(\mathfrak{g})$-module by adding (\ref{oneact}).

\textbf{Type $E_{6}$}: If $3|n$,   define:
$$\rho:\;G\To \k,\;\; \left\{ \begin{array}{ll}
K_{i}\mapsto 1, & \mathrm{if}\ i=2,4,\\
K_{i}\mapsto \q^{\frac{n}{3}}, & \mathrm{if}\ i=1,5,\\
K_i\mapsto \q^{\frac{2n}{3}}, & \mathrm{if}\ i=3,6.
\end{array}\right.$$
Then $\rho$ extends to a ${\Q}\mathbf{u}_{q}(\mathfrak{g})$-module by adding (\ref{oneact}).

\textbf{Type $E_{7}$}: If $4|n$,   define:
$$\rho:\;G\To \k,\;\; \left\{ \begin{array}{ll}
K_{i}\mapsto 1, & \mathrm{if}\ i=1,3,4,6,\\
K_i\mapsto \q^{\frac{n}{4}}, & \mathrm{if}\ i=2,5,7.
\end{array}\right.$$
Then $\rho$ extends to a ${\Q}\mathbf{u}_{q}(\mathfrak{g})$-module by adding (\ref{oneact}).

\end{proof}

\begin{remark} \emph{(1) Except type $G_{2}$, Etingof and Gelaki proved that $D(A_{q}(\mathfrak{g}))$ is always twist equivalent to
a Hopf algebra provided $n$ is odd and $(n,|(a_{ij})|)=1$ (for type $G_2$, they need one more condition, that
is, $3\nmid n$), where $|(a_{ij})|$ is the determinant of the Cartan matrix. It is well-known that
 the determinant of the Cartan matrix of type $A_m$ (resp. $E_6$) is $m+1$ (resp. $3$). Therefore, our results imply that the condition `` $n$ is odd and $(n,|(a_{ij})|)=1$" can not be removed in general. One could even ask whether such a condition is a necessary condition. But it is not. For example, let $n=2$. One can use Corollary \ref{c1.7} and Lemma \ref{l1.5} to show that $\phi$ is already a coboundary in $A_{q}(\mathfrak{g})$ and thus
 $A_{q}(\mathfrak{g})$ is twist equivalent to a Hopf algebra. Therefore, $D(A_{q}(\mathfrak{g}))$ is twist equivalent to a Hopf algebra too.}

 \emph{(2) Our methods cannot be applied to Lie algebras of type $E_{8}, F_{4}$ and $G_{2}$. We do not know whether there is an $n$ such that
 ${\Q}\mathbf{u}_{q}(\mathfrak{g})$ is genuine when $\mathfrak{g}$ is one of these types. }
\end{remark}

\begin{problem} \emph{(1) For type $E_{8}$, is ${\Q}\mathbf{u}_{q}(\mathfrak{g})$  twist equivalent to a Hopf algebra?
For type $F_{4}$, is ${\Q}\mathbf{u}_{q}(\mathfrak{g})$ genuine when $4|n$? For type $G_{2}$, is ${\Q}\mathbf{u}_{q}(\mathfrak{g})$ genuine when $6|n$?}

 \emph{(2) Give a complete list of genuine quasi-Frobenius-Lusztig kernels.}

 \emph{(3) How can one determine whether a given finite dimensional quasi-Hopf algebra $H$ over $\k$ is twist equivalent to some Hopf algebra or not?}
\end{problem}

\section*{Acknowledgments}
The first author would like thank Department of Mathematics, University
 of Antwerp for its hospitality. This work is supported by the NSF of China (No. 11071111) and by a FWO grant.

\end{document}